\documentclass[a4paper,reqno]{amsart}
\usepackage{amsmath,amssymb,amsthm,graphicx,mathrsfs,url}
\usepackage[usenames,dvipsnames]{color}
\definecolor{darkred}{rgb}{0.4,0.1,0.1}
\definecolor{darkblue}{rgb}{0.1,0.1,0.4}

\usepackage{tikz}
\usepackage{tikz-3dplot}
\usetikzlibrary{datavisualization}
\usetikzlibrary{datavisualization.formats.functions}
\usetikzlibrary{hobby,backgrounds,patterns}
\usepackage[normalem]{ulem}

\numberwithin{equation}{section}
\theoremstyle{plain}

\newtheorem{theorem}{Theorem}[section]
\newtheorem{lemma}[theorem]{Lemma}

\newtheorem{proposition}[theorem]{Proposition}
\newtheorem{corollary}[theorem]{Corollary}

\theoremstyle{remark}
\newtheorem{remark}[theorem]{Remark}

\theoremstyle{definition}
\newtheorem{example}[theorem]{Example}
\newtheorem{definition}[theorem]{Definition}
\newtheorem{hypothesis}[theorem]{Hypothesis}

\newcommand\cH{\mathcal H}

\newcommand\cN{\mathcal N}

\newcommand\eps{\varepsilon}

\DeclareMathOperator{\spann}{span}
\DeclareMathOperator{\diver}{div}
\DeclareMathOperator{\curl}{curl}

\definecolor{darkgreen}{rgb}{0.1,0.45,0.1}
\definecolor{darkblue}{rgb}{0.1,0.1,0.4}
\definecolor{darkgrey}{rgb}{0.5,0.5,0.5}

\definecolor{darkred}{rgb}{0.6,0.0,0.0}

\newcommand\void[1]{}

\def\eps{\varepsilon}

\DeclareMathOperator\argmin{arg\,min}

\renewcommand{\phi}{\varphi}

\def\sa{\mathfrak a}

   \def\cH{{\mathcal H}}   
      
   \def\cN{{\mathcal N}}   \def\cO{{\mathcal O}}

\def\cV{{\mathcal V}}

\def\R{\mathbb{R}}
\def\C{\mathbb{C}}

\def\N{\mathbb{N}}

\newcommand{\dom}{\mathrm{dom}\,}

\newcommand{\e}{\textup{e}}

\allowdisplaybreaks

\def\ee{{\mathrm e}}

\newcounter{counter_a}

\title[On the hot spots conjecture in higher dimensions]{On the hot spots conjecture in higher dimensions}

 \author[J.B.~Kennedy]{James B.~Kennedy}
 \address{Departamento de Matem\'atica, Faculdade de Ci\^encias, Universidade de Lisboa, Campo Grande, Edif\'icio C6, 1749-016 Lisboa, Portugal {\rm and} Grupo de F\'isica Matem\'atica, Instituto Superior T\'ecnico, Av.\ Rovisco Pais, 1049-001 Lisboa, Portugal}
 \email{jbkennedy@ciencias.ulisboa.pt}

 \author[J.~Rohleder]{Jonathan Rohleder}
 \address{Matematiska institutionen \\ Stockholms universitet \\
 106 91 Stockholm \\
 Sweden}
 \email{jonathan.rohleder@math.su.se}

\thanks{The authors wish to thank Serge Nicaise for his valuable help with the literature on Maxwell's equation and curl curl operators. J.B.K.\ extends his thanks for the hospitality afforded him during a visit to Stockholm University, where part of this work was completed. The work of J.B.K.\ was partly supported by the Funda\c{c}\~ao para a Ci\^encia e a Tecnologia, Portugal, within the scope of the project SpectralOPs, reference 2023.13921.PEX, and grant UIDB/00208/2020 (DOI 10.54499/UIDB/00208/2020). The work of J.R.\ was partly supported by grant no.\ 2022-03342 of the Swedish Research Council (VR)
}

\keywords{Laplace operator, Neumann boundary conditions, mixed boundary conditions, eigenvalue inequality, polyhedral domain, Lipschitz domain}
\subjclass[2020]{Primary 35J05, 35J25; secondary 35J50, 35J57, 35P15.}

\begin{document}

\begin{abstract}
We prove a strong form of the hot spots conjecture for a class of domains in $\R^d$ which are a natural generalization of the lip domains of Atar and Burdzy [J.\ Amer.\ Math.\ Soc.\ 17 (2004), 243--265] in dimension two, as well as for a class of symmetric domains in $\R^d$ generalizing the domains studied by Jerison and Nadirashvili [J.\ Amer.\ Math.\ Soc.\ 13 (2000), 741--772]. Our method of proof is based on studying a vector-valued Laplace operator whose spectrum contains the spectrum of the Neumann Laplacian. This proof is essentially variational and does not require tools from stochastic analysis, nor does it use deformation arguments. In particular, it contains a new proof of the main result of Jerison and Nadirashvili.
\end{abstract}

\maketitle

\section{Introduction}

Our goal in this paper is to prove the \emph{hot spots conjecture} of Rauch for certain classes of domains in arbitrary space dimension. We recall that this conjecture, in perhaps its most common form, asserts that any eigenfunction $\psi_2$ associated with the first nontrivial eigenvalue $\mu_2 > 0$ of the Neumann Laplacian,
\begin{displaymath}
\begin{aligned}
	-\Delta \psi_2 &= \mu_2 \psi_2 \qquad &&\text{in } \Omega,\\
	\frac{\partial \psi_2}{\partial\nu} &=0 &&\text{on } \partial\Omega,
\end{aligned}
\end{displaymath}
on a bounded, sufficiently regular domain $\Omega \subset \R^d$ (with boundary $\partial \Omega$ and outer unit normal $\nu$), attains its minimum and its maximum only on $\partial\Omega$. The name \emph{hot spots} stems from the significance of $\psi_2$ for the heat equation via Fourier's method; the intuition is that, in the case of a perfectly insulated body (Neumann conditions), the hottest and coldest spots in the body will migrate to the surface of the body for large times.

We recall that the conjecture, dating to the 1970s, is known to be true in $\R^2$ for certain special classes of convex domains, in particular triangles \cite{JM20,JM20err}, as well as ``long and thin'' domains \cite{AB04,BB99}, as well as certain domains symmetric with respect to both axes of symmetry \cite{JN00} and convex domains with one axis of symmetry \cite{P02}, although it is still open for convex planar domains in general. However, counterexample domains (with holes) have been known for around a quarter of a century, at least in dimension two \cite{BW99}. For further recent developments around the hot spots conjecture we refer to \cite{JM22,K21,KT19,S20,S23}. In particular, it was shown very recently in \cite{DP24} that even on convex domains the conjecture can be false in high dimension. The corresponding problem for the Laplacian with mixed Neumann-Dirichlet boundary conditions has been extensively studied in recent years in the case of planar domains in \cite{BP04,BPP04,H24+,H24a,H24b}; a special case of it in higher dimensions will also be included in our results here. Analogous problems for Laplacians on discrete and metric graphs were considered in \cite{GP19,KR21,LS24+}.

While most of the proofs to date either rely on fine estimates for reflected Brownian motion, or else make heavy use of symmetries, a new variational approach was recently introduced 
\cite{Rprep}. This approach uses only ``classical'' (and deterministic) tools from the theory of PDEs, and is based on finding a Laplacian operator acting on vector fields, whose spectrum decomposes neatly into the nonzero eigenvalues of the Neumann and the Dirichlet Laplacians; its first eigenvalue is $\mu_2$ and its corresponding eigenfunction is, essentially, $\nabla \psi_2$. Studying positivity properties of this eigenfunction allows one to provide a simple proof of the results of \cite{AB04} on so-called ``lip domains'' (that is, domains enclosed by the graphs of two Lipschitz functions with Lipschitz constant at most one): on these domains, after a suitable rotation, as long as the domain is not a rectangle, $\mu_2$ is simple, and $\partial_1 \psi_2$ and $\partial_2\psi_2$ are positive functions, meaning, in particular, that $\psi_2$ attains its minimum and maximum only on the boundary, essentially at the corner points where the two aforementioned functions which constitute $\partial \Omega$ meet.

As noted above, our goal is to study the problem in higher dimensions, using an enhancement of the variational approach in \cite{Rprep}. While the conjecture still makes sense in any space dimension, to date, very few positive results seem to be known; except for an observation on cylindrical domains \cite{K85} to the best of our knowledge the only non-trivial domains for which the conjecture is known to be true are in \cite{CLW19,Y11}; the former uses a deformation technique to prove the conjecture for a specific generalization of the symmetric domains of Jerison and Kenig, while the latter applies Brownian motion techniques to a generalization of the lip domains of Atar and Burdzy (see Section~\ref{sec:lip} for more details).

We will push forward the variational approach of \cite{Rprep} to higher dimensions to obtain positive results for two classes of domains. The first, for which we will obtain the strongest possible form of the hot spots conjecture, contains the two-dimensional lip domains of \cite{AB04}  (modulo an additional mild regularity assumption), and the higher dimensional lip domains of \cite{Y11}, as special cases; we will argue (again, see Section~\ref{sec:lip}) that this new class of domains, while somewhat restrictive in higher dimensions, is nevertheless a ``natural'' choice for all the partial derivatives of the eigenfunction $\psi_2$ to be non-sign changing, due to the Neumann condition; cf.\ Remark \ref{rem:maximalClass}. The second class, for which we will obtain a slightly different result, for certain eigenfunctions of symmetric domains, contains the two-dimensional symmetric domains of \cite{JN00} as a special case and in particular provides a new proof of \cite[Theorem~1.4]{JN00}, up to slightly more restrictive assumptions on the regularity of the boundary.

In order to formulate our main results we start with a technical assumption.

\begin{definition}
\label{def:piecewise-smooth}
We say $\Omega \subset \R^d$, $d \geq 2$, is a \emph{piecewise smooth} domain if:
\begin{enumerate}
\item It is a bounded, open, connected set with Lipschitz boundary; and
\item there exists an exceptional set $\Sigma \subset \partial\Omega$, of zero $d-1$-dimensional Hausdorff measure, such that $\partial\Omega \setminus \Sigma$ consists of a finite number of connected components, relatively open in $\partial\Omega$, each of which is $C^\infty$ (that is, for any $z \in \partial\Omega \setminus \Sigma$, there exists an open neighborhood $U_z \ni z$ such that $\partial\Omega \cap U_z$ is a $C^\infty$-manifold).
\end{enumerate}
\end{definition}

By way of analogy with polyhedra, which are immediately seen to be special cases, for a piecewise smooth domain $\Omega$ we will refer to the (closures of the) connected components of $\partial\Omega \setminus \Sigma$ as the \emph{faces} of its boundary $\partial\Omega$. Points in $\partial \Omega \setminus \Sigma$ will occasionally be called {\em regular points}.

\begin{hypothesis}
\label{hyp:dom}
The domain $\Omega \subset \R^d$, $d \geq 2$, is piecewise smooth and satisfies a uniform exterior ball condition.
\end{hypothesis}

Domains satisfying Hypothesis~\ref{hyp:dom} have the special property that the domain of the Neumann Laplacian on $L^2(\Omega)$ is contained in the Sobolev space $H^2(\Omega)$ of order two, 
as follows immediately from \cite[Introduction, example (ii) and Proposition~4.8]{AGMT10}. In addition to this requirement, we will make use of curvatures of the boundary at its regular points, which requires certain regularity. We note that slightly different hypotheses could be assumed, in particular, the regularity condition could certainly be weakened from $C^\infty$ to $C^3$.

\begin{definition}
\label{def:lip}
We will call a Lipschitz domain $\Omega \subset \R^d$, $d\geq 2$, a \emph{lip domain} if, possibly after a rotation, for almost all $x \in \partial\Omega$ the outer unit normal $\nu(x)$ to $\Omega$ at $x$ either has exactly two nonzero components and these have opposite signs, or $\nu(x) = \pm e_j$ for some standard basis vector $e_j$.
\end{definition}

By default we will always assume our domain to be rotated in this fashion, unless explicitly stated otherwise. We note again that if $d=2$, then rotating a lip domain in Definition \ref{def:lip} clockwise by an angle of $\pi/4$ leads to the original notion of two-dimensional lip domains as used in \cite{AB04, Rprep}. Our class of lip domains is also larger than the class considered in \cite{Y11}, see Remark~\ref{rem:Y11}.

\begin{theorem}
\label{thm:hot-spots-I}
Let $\Omega \subset \R^d$, $d \geq 2$, be a bounded lip domain satisfying Hypothesis~\ref{hyp:dom}, and let $\psi_2$ be any eigenfunction of the Neumann Laplacian on $\Omega$ corresponding to $\mu_2$. Then $\psi_2$ attains its maximum and minimum in $\overline{\Omega}$ only on $\partial \Omega$. Furthermore, if $\Omega$ is rotated as described in Definition~\ref{def:lip}, then in each coordinate direction, $\psi_2$ is either strictly monotonic or constant. More precisely, for each $i \in \{1, \dots, d\}$ one has $\partial_i \psi_2 > 0$ in $\Omega$, or $\partial_i \psi_2 < 0$ in $\Omega$, or $\partial_i \psi_2 = 0$ identically.
\end{theorem}

Note that under these assumptions $\mu_2$ need not be simple (and it is possible that $\partial_i \psi_2 = 0$): beyond just rectangular prisms, product domains of the form $D \times I$ ($D \subset \R^{d-1}$ a lip domain, $I$ an interval) are included as special cases. We will discuss this, along with several other examples which are not of this form, in Section~\ref{sec:lip}.

\begin{remark}
The uniform exterior ball condition required in Hypothesis \ref{hyp:dom} is not just a technical assumption in Theorem \ref{thm:hot-spots-I}. In fact, already in dimension $d = 2$, any simply connected, not necessarily convex, polygon with all sides axioparallel is a lip domain with piecewise smooth boundary. However, it is not to be expected in general that on such domains $\psi_2$ is monotonic in every coordinate direction if, e.g., such $\Omega$ is ``spiral-shaped''. Our definition of lip domains without the exterior ball condition would include non-simply connected domains, where counterexamples to the hot spots conjecture in the spirit of \cite{BW99}, e.g.\ for polygons with axioparallel sides, could be constructed. We do not go into details.
\end{remark}

We next proceed to the case of symmetric domains.

\begin{theorem}
\label{thm:hot-spots-II}
Let $\Omega \subset \R^d$, $d \geq 2$, in addition to satisfying Hypothesis~\ref{hyp:dom}, have a reflection symmetry in every coordinate plane $\{x_i=0\}$, $i=1,\ldots,d$, and assume there exists at least one (open) orthant $\mathcal{O}$ such that $\Omega \cap \mathcal{O}$ is a lip domain in the sense of Definition~\ref{def:lip}. For any $j=1,\ldots,d$, up to scalar multiples, there is a unique eigenfunction $\psi_j$ associated with the smallest eigenvalue among all eigenfunctions which are reflection antisymmetric in the plane $\{x_j=0\}$ and symmetric in every other coordinate plane $\{x_i=0\}$, $i \neq j$. Moreover, for each $i \in \{1,\ldots,d\}$ we have the trichotomy $\partial_i \psi_j >0$ in $\Omega \cap \mathcal{O}$, or $\partial_i \psi_j < 0$ in $\Omega \cap \mathcal{O}$, or $\partial_i \psi_j = 0$ in $\Omega \cap \mathcal{O}$, and the maximum and minimum of $\psi_j$ in $\overline{\Omega}$ are attained only on $\partial\Omega$.
\end{theorem}

As mentioned above, this theorem includes the domains of \cite{JN00} as a special case when $d=2$, up to boundary regularity. Indeed, in dimension two, under Hypothesis \ref{hyp:dom} the assumption that $\Omega \cap \cO$ is a lip domain for some orthant $\cO$ is equivalent to the assumption of \cite[Theorem 1.1]{JN00} that all the vertical and horizontal cross sections of $\Omega$ are intervals. Moreover, in dimension two a simple argument using nodal domain counts shows that, up to the right choice of basis, any eigenfunction associated with $\mu_2$ has the antisymmetry/symmetry properties demanded by the theorem; thus the theorem is applicable and implies the hot spots conjecture in this case. We do not expect this to be true in general for $d \geq 3$.  Our class partly overlaps with the class considered in \cite{CLW19}, but in general will be distinct from it, see Remark~\ref{rem:CLW19}.

In dimension $3$, a simple example of a domain satisfying the assumptions of Theorem~\ref{thm:hot-spots-II} but not Theorem~\ref{thm:hot-spots-I} is any octahedron with a rectangular base, whose two apices sit directly above (or below) the center of the base.

Note that in general the trichotomy for $\partial_i\psi_j$ in the theorem will hold orthant-wise, except for $\partial_j\psi_j$, which will be strictly monotonic in $\Omega$ due to the antisymmetry and the fact that no nonzero Laplacian eigenfunction can satisfy $\psi_j = \partial_j\psi_j = 0$ identically on the plane $\{x_j=0\}$ (by unique continuation). The disk is already an example: if we choose the eigenfunction $\psi_2$ to be zero in the plane $\{x_2 = 0\}$ and positive in the upper half-plane $\{x_2 > 0\}$, then $\partial_1 \psi_2$ will be negative in the first and third quadrants, and positive in the second and fourth quadrants.

We will obtain both theorems as consequences of a slightly more general result, namely for the Laplacian with Neumann conditions except possibly on a special face, which may have Dirichlet boundary conditions.

\begin{theorem}
\label{thm:hot-spots-III}
Let $\Omega \subset \R^d$, $d \geq 2$, be a bounded lip domain satisfying Hypothesis~\ref{hyp:dom}. We assume the following boundary conditions on $\partial\Omega$:
\begin{enumerate}
\item[(i)] Dirichlet boundary conditions on \emph{at most} one face $\Gamma_{\rm D}$ of $\partial\Omega$, as long as the face is flat and perpendicular to a coordinate axis;
\item[(ii)] Neumann boundary conditions on all other faces of $\partial\Omega$, whose union we denote by $\Gamma_{\rm N}$.
\end{enumerate}
Let $\lambda>0$ be the smallest nonzero eigenvalue of the Laplacian with the above boundary conditions, and let $\psi$ be any associated eigenfunction. Then, for each $i \in \{1,\ldots,d\}$, we have the trichotomy $\partial_i \psi >0$ in $\Omega$, or $\partial_i \psi < 0$ in $\Omega$, or $\partial_i \psi = 0$ identically in $\Omega$.
\end{theorem}

Most of the paper will be devoted to the proof of Theorem~\ref{thm:hot-spots-III}. As mentioned above, the method of proof is based on studying a vectorial Laplacian which contains all the Neumann Laplacian eigenvalues, in the spirit of \cite{Rprep}; this is the subject of Section~\ref{sec:operator}. This operator is also the subject of Section~\ref{sec:curl-curl}, where its particular form in dimension three is studied and used to obtain an inequality between the smallest nontrivial eigenvalue of the Laplacian in Theorem \ref{thm:hot-spots-III} and the smallest eigenvalue of a curl curl operator, previously obtained for a different class of domains in \cite{P15,Z18}.

The actual proof of Theorem \ref{thm:hot-spots-III} will be given in Section~\ref{sec:proof}. Theorem~\ref{thm:hot-spots-I} is already contained in Theorem~\ref{thm:hot-spots-III}, while Theorem~\ref{thm:hot-spots-II} will follow from a rather standard argument using the symmetry properties of eigenfunctions of symmetric domains, given in Section~\ref{sec:symm}.

\section{On lip domains in dimension $d$}
\label{sec:lip}

We start by illustrating what kinds of domains are covered by Theorem~\ref{thm:hot-spots-I} based on a few examples and classes of examples, as well as a comparison to other kinds of lip domains considered in the literature, in both higher dimensions \cite{Y11} and two dimensions \cite{AB04,Rprep}.

We start by observing that any cylindrical domain of the form $D \times I$, where $D \subset \R^{d-1}$ is a lip domain, is also a lip domain.

\begin{example}
Suppose $D \subset \R^{d-1}$ is a lip domain in the sense of Definition~\ref{def:lip} and $I \subset \R$ is a bounded interval. Then $\Omega := D \times I \subset \R^{d}$ is a lip domain in $\R^d$. To see this, first observe that $\partial\Omega$ decomposes naturally as $\partial\Omega = (\partial D \times \overline{I}) \cup (\overline{D} \times \partial I)$. Clearly the two faces $\overline{D} \times \partial I$ are perpendicular to the $x_d$-axis, with $\nu(x) = \pm e_d$ there. On $\partial D \times \overline{I}$, by assumption on $D$, up to a rotation, the normal vector will have the form $(\tilde\nu,0)^\top$, where $\tilde\nu$ is a $d-1$-dimensional vector which is, by assumption, either some $\pm e_j$ for $j\in \{1,\ldots,d-1\}$, or has exactly two nonzero components, of opposite signs. Thus the same is true of $\nu$, and $\Omega$ is a lip domain.

More generally, if $\Omega_1 \subset \R^{d_1}$ and $\Omega_2 \subset \R^{d_2}$ are lip domains, then $\Omega_1 \times \Omega_2 \subset \R^{d_1 + d_2}$ is as well, by the same reasoning. We also observe that if $\Omega_1$ and $\Omega_2$ satisfy Hypothesis~\ref{hyp:dom}, then clearly $\Omega = \Omega_1 \times \Omega_2$ does too.
\end{example}

We next give a prototypical example of a domain which is not such a cylinder; for comprehensibility, we restrict to three dimensions.

\begin{example}
\label{example:double-triangular-prism}
A simple example of a polyhedral lip domain in $\R^3$ which is not a cylinder is as follows. We take two differently oriented triangular prisms,
\begin{displaymath}
\begin{aligned}
	T_x &:= \{ (x,y,z) \in \R^3: 0 \leq x \leq 1,\, 0 \leq y \leq 1,\, 0 \leq z \leq x\},\\
	T_y &:= \{ (x,y,z) \in \R^3: 0 \leq x \leq 1,\, 0 \leq y \leq 1,\, y-1 \leq z \leq 0\},
\end{aligned}
\end{displaymath}
so that $T_x \cap T_y$ is the square with vertices at $(0,0,0)$, $(1,0,0)$, $(0,1,0)$, $(1,1,0)$ in the $xy$-plane.
\begin{figure}[ht]
\tdplotsetmaincoords{100}{55}
\begin{tikzpicture}[scale=2,line join=bevel,tdplot_main_coords]
\coordinate (P1) at (0,0,0);
\coordinate (P2) at (1,0,0);
\coordinate (P3) at (1,1,0);
\coordinate (P4) at (0,1,0);
\coordinate (U1) at (1,0,1);
\coordinate (U2) at (1,1,1);
\coordinate (L1) at (0,0,-1);
\coordinate (L2) at (1,0,-1);

\draw[thick] (P3) -- (U2) -- (P4) -- cycle;
\draw[thick] (P1) -- (P4) -- (U2) -- (U1) -- cycle;

\draw[thick] (L1) -- (L2) -- (P3) -- (P4) -- cycle;
\draw[thick] (P1) -- (P4) -- (L1) -- cycle;

\draw[thick,dashed] (L1) -- (P1) -- (U1) -- (L2) -- cycle;

\draw[fill opacity=0.6,fill=blue!70!black] (P1) -- (P4) -- (U2) -- (U1) -- cycle;
\draw[fill opacity=0.6,fill=orange!80!black] (U2) -- (P3) -- (P4) -- cycle;
\draw[fill opacity=0.6,fill=green!50!black] (L1) -- (L2) -- (P3) -- (P4) -- cycle;
\draw[fill opacity=0.6,fill=purple!70!black] (P1) -- (P4) -- (L1) -- cycle;

\draw[thick,->] (-0.8,-0.8,0) -- (-0.8,-0.8,0.5);
\node at (-0.8,-0.8,0.5) [anchor=south] {$z$};
\draw[thick,->] (-0.8,-0.8,0) -- (-0.3,-0.8,0);
\node at (-0.3,-0.8,0) [anchor=south] {$x$};
\draw[thick,->] (-0.8,-0.8,0) -- (-0.8,-0.3,0);
\node at (-0.8,-0.3,0) [anchor=north] {$y$};
\end{tikzpicture}
\caption{The domain $\Omega$. The upper triangular prism is $T_x$, the lower is $T_y$.}
\label{fig:3dLip}
\end{figure}
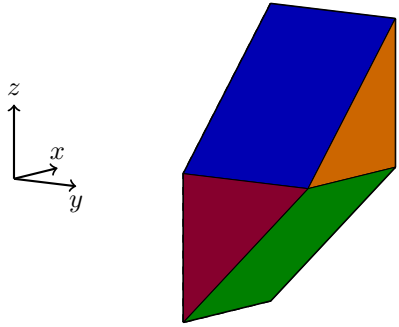
Denoting by $\Omega$ the interior of $T_x \cup T_y$, see Figure \ref{fig:3dLip}, we see that $\Omega$ is a polyhedron with six faces (four quadrilaterals and two triangles) with four oriented parallel to two axes each and the remaining two, slanted, having outer normal vectors pointing towards $(-1,0,1)$ (in the case of the upper face, part of $T_x$) and, respectively, $(0,1,-1)$ (in the case of the lower face, part of $T_y$). Thus $\Omega$ is a lip domain (and obviously satisfies Hypothesis~\ref{hyp:dom}); but it is clearly not the cross product of an interval with a fixed two-dimensional domain.

One could easily construct variants, for example, by modifying the slope of the two slanting faces, or by shifting $T_x$ upwards, say, to become
\begin{displaymath}
	T_{x,c} := \{ (x,y,z) \in \R^3: 0 \leq x \leq 1,\, 0 \leq y \leq 1,\, c \leq z \leq x\}
\end{displaymath}
for some given $c>0$ and inserting a rectangular prism of the form
\begin{displaymath}
	P_c := \{(x,y,z) \in \R^3: 0 \leq x \leq 1,\, 0 \leq y \leq 1,\, 0 \leq z \leq c\};
\end{displaymath}
then if $\Omega_c$ is the (interior of) $T_{x,c} \cup P_c \cup T_y$, it will still be a non-cylindrical polyhedral lip domain as before.

A further variant can be constructed by truncating $\Omega$ along a vertical plane of the form $\{x-y+C=0\}$ for a suitable constant $C$, to create a domain $\Omega_C$ which now has a face with normal vector pointing towards $(-1,1,0)$. One could alternatively take a smooth surface whose normal points into the same octant in place of the plane.
\end{example}

\begin{remark}
\label{rem:Y11}
In \cite{Y11}, Yang gives what is possibly the first proof of the hot spots conjecture for a class of (non-cylindrical) domains in $\R^d$, which in \cite{Y11} are called lip domains, as they reduce to the lip domains of \cite{AB04} when $d=2$. This class is, however, significantly smaller than the class covered by Definition~\ref{def:lip} and Hypothesis~\ref{hyp:dom} (although Yang only insists that the domains be piecewise-$C^1$ rather than piecewise-$C^\infty$); essentially, in \cite{Y11} there exists a ``distinguished'' direction $x_1$, so that the domains consist of $2d$ faces, two of which have normal pointing in $\pm e_1$, while on the other faces the normal vector should always have nonzero $x_1$-component, and one other nonzero component (with opposite sign).

More precisely, in the notation of \cite[Lemma~2.1]{Y11}, the cases (1) and (2) are exactly $\pm e_1$; while in case (3), the normal vector is of the form $F_{i,1}^-(x_1)e_1 - F_{i,2}^- (x_1)e_i$ for nonnegative functions $F_{i,1}^-$ and $F_{i,2}^-$ of $x_1$ (since $f_i'$ does not change sign), and in case (4), the vector is now $-F_{i,1}^+(x_1)e_1 + F_{i,2}^+ (x_1)e_i$ for nonnegative $F_{i,1}^+$ and $F_{i,2}^+$. Thus, up to the question of regularity, each domain in \cite{Y11} is a lip domain in our sense.

On the other hand, the domains $\Omega_C$ from Example~\ref{example:double-triangular-prism} cannot fit into the framework of \cite{Y11}, as they can have up to seven affine faces, while those in \cite{Y11} necessarily have six (smooth) faces.
\end{remark}

However, if $d=2$, it is immediate that (up to the regularity assumption imposed on the faces) our definition reduces to the one in \cite{AB04}, see the introduction to \cite{Rprep}.

\begin{remark}
\label{rem:CLW19}
In \cite{CLW19}, the only other paper of which we are aware where the hot spots conjecture is proved for a non-trivial family of domains in higher dimensions, the domains considered generalize those of \cite{JN00} in a particular way to higher dimensions. Our Theorem~\ref{thm:hot-spots-II} provides a partly overlapping but distinct generalization: for example, both cover regular octahedra in $\R^3$. But while \cite{CLW19} requires additional rotation symmetries which we do not (their condition (2)), our lip domain condition imposes a stronger restriction on the possible form of the smooth faces than their condition (3).
\end{remark}

\begin{remark}
\label{rem:maximalClass}
Our definition of a multidimensional lip domain appears natural in the context of the hot spots problem. It is tailor-made for the proof strategy suggested first in the case $d = 2$ in \cite{R24prep}, cf.\ Lemma \ref{lem:sublattice}, and covers a very broad class of domains for which it is possible that an eigenfunction $\psi$ of the Neumann Laplacian can  have all its partial derivatives non-sign changing. Roughly speaking, if $\partial \Omega$ e.g.\ contains regular points $p, q$ with 
\begin{align*}
 \nu (p) = (\alpha (p), - \beta (p), 0, \dots, 0)^\top \quad \text{and} \quad \nu (q) = (\alpha (q), \beta (q), 0, \dots, 0)^\top
\end{align*}
with positive $\alpha (p), \alpha (q), \beta (p), \beta (q)$, then the Neumann boundary conditions
\begin{align*}
 \alpha (p) \partial_1 \psi (p) + \beta (p) \partial_2 \psi (p) = 0
\end{align*}
and 
\begin{align*}
 \alpha (q) \partial_1 \psi (q) + \beta (q) \partial_2 \psi (q) = 0
\end{align*}
imply that one of $\partial_1 \psi, \partial_2 \psi$ should change sign on $\partial \Omega$ and, hence, inside $\Omega$.
\end{remark}

\section{An auxiliary operator}
\label{sec:operator}

In this section we define an auxiliary operator which acts as the Laplacian on vector fields and has the gradients of Laplacian eigenfunctions (with Neumann or certain mixed boundary conditions) among its eigenfields. Its definition requires a little bit of geometry.

Let $\Omega \subset \R^d$ satisfy Hypothesis~\ref{hyp:dom}; we assume that $\partial\Omega = \Gamma_{\rm N} \cup \Gamma_{\rm D}$, where $\Gamma_{\rm N}$, $\Gamma_{\rm D}$ are closed, connected subsets of $\partial \Omega$ whose intersection has zero $(d-1)$-dimensional Hausdorff measure. Note that $\Gamma_{\rm D} = \emptyset$ is allowed.

We first consider the form domain
\begin{align}\label{eq:H}
 V := \left\{ u \in H^1 (\Omega)^d : \langle u |_{\partial \Omega}, \nu \rangle = 0~\text{on}~\Gamma_{\rm N},\,
	|\langle u|_{\partial\Omega},\nu\rangle| = |u |_{\partial \Omega}|~\text{on}~\Gamma_{\rm D} \right\},
\end{align}
where here and throughout $\langle \cdot, \cdot \rangle$ denotes the inner product in $\C^d$, applied pointwise almost everywhere. That is, we consider vector fields with components in the Sobolev space $H^1 (\Omega)$ whose trace is normal to the boundary almost everywhere on $\Gamma_{\rm D}$, and tangential almost everywhere on $\Gamma_{\rm N}$. Our prototype will be $u = \nabla \psi$, where $\psi$ is any (non-constant) eigenfunction of the Laplacian with Neumann conditions on $\Gamma_{\rm N}$ and Dirichlet conditions on $\Gamma_{\rm D}$; note that Hypothesis~\ref{hyp:dom} guarantees that in fact $\nabla \psi \in H^1(\Omega)^d$, which is not true on an arbitrary Lipschitz domain.

We will introduce the operator via the associated sesquilinear form. To this end, we start by noting that if $\Omega$ is piecewise smooth in the sense of Definition~\ref{def:piecewise-smooth}, then at almost every point $p \in \partial \Omega$ there exists a shape operator $L = L_p$ that, in a local parametrization $f : \R^{d - 1} \supset U_j \to \R^d$ with $U_j$ open and $f (x) = p$, can be defined by the relations
\begin{align*}
 L_p \partial_j f (x) = - \partial_j \nu (x), \quad j = 1, \dots, d - 1,
\end{align*}
as a linear, symmetric map on the tangential space at the point $f (p)$, where we identify the normal vector field $\nu$ on $\partial \Omega$ with the corresponding Gaussian map given locally by $\nu \circ f$. For our purposes it will be convenient to consider $L_p$ as a linear mapping $\R^d \to \R^d$, which we do through the trivial extension in the normal direction, that is, $L_p \nu (x) = 0$.

We can now define
\begin{equation*}
 \sa [u, v] = \sum_{j = 1}^d \int_\Omega \langle \nabla u_j, \nabla v_j\rangle - \int_{\partial \Omega} \langle L u, v \rangle, \quad \dom \sa = V;
\end{equation*}
here and henceforth all integration over $\Omega$ is with respect to the $d$-dimensional Lebesgue measure and integration over $\partial \Omega$ is with respect to the $d-1$-dimensional Hausdorff measure. We point out that the matrix function $p \mapsto L_p$, defined almost everywhere on $\partial \Omega$, is bounded and, hence, $\sa$ is well-defined.

We then denote by $A$ the associated operator on $L^2(\Omega)^d$, that is,
\begin{displaymath}
\begin{aligned}
	\dom A &= \left\{ u \in V: \text{there exists } w \in L^2(\Omega)^d \text{ s.t.\ } \sa [u,v] = \int_\Omega \langle w, v \rangle, v \in V \right\},\\
	Au &= w,
\end{aligned}
\end{displaymath}
and by $\eta_k = \eta_k (A)$, $k = 1, 2, \dots$, its ordered eigenvalues, counted with their multiplicities, which makes sense due to the following result.

\begin{proposition}
\label{prop:A-general}
Let $\Omega \subset \R^d$ be a bounded, connected Lipschitz domain with piecewise smooth boundary. Then $A$ is a self-adjoint operator on $L^2(\Omega)^d$ with compact resolvent. Its spectrum consists of a discrete sequence of real eigenvalues bounded from below.
\end{proposition}

\begin{proof}
We will prove that the sesquilinear form $\sa$ is symmetric and bounded from below, i.e., there exists some $\mu \in \R$ such that
\begin{align}\label{eq:semibounded}
 \sa [u, u] & \geq \mu \int_\Omega |u|^2, \quad u \in \dom \sa,
\end{align}
holds. Moreover, we will show that $V = \dom \sa$, equipped with the norm
\begin{align}\label{eq:anorm}
 \| u \|_\sa := \left( \sa [u, u] + (1 - \mu) \int_\Omega |u|^2 \right)^{\frac{1}{2}},
\end{align}
is complete. As $\dom \sa \supset C_0^\infty (\Omega)^d$ is dense in $L^2 (\Omega)^d$, it follows from \cite[Chapter VI, § 1--2]{Kato} that the operator $A$ given above is well-defined and self-adjoint in $L^2 (\Omega)^d$ and its spectrum is bounded from below by $\mu$.

First of all, it follows from the symmetry of the shape operator $L_p$ that $\sa$ is symmetric. To show semi-boundedness, recall from, e.g., \cite[Lemma 4.2]{GM09} the following refined trace norm estimate: for every $\eps > 0$ there exists some $\beta (\eps) > 0$ such that
\begin{align}\label{eq:GM}
 \int_{\partial \Omega} |u|^2 \leq \eps \sum_{j = 1}^d \int_\Omega |\nabla u_j|^2 + \beta (\eps) \int_\Omega |u|^2, \quad u \in H^1 (\Omega)^d.
\end{align}
Moreover, we will use the fact that, on each compact face $\Upsilon$ of $\partial \Omega$, there exists a constant $\kappa_{\max,\Upsilon} \geq 0$ such that
\begin{align}\label{eq:boundedCurvature}
 \langle L_p \xi, \xi \rangle \leq \kappa_{\max,\Upsilon} |\xi|^2, \quad \xi \in \C^d,
\end{align}
holds for almost all $p \in \Upsilon$, i.e.\ uniformly in $p$; since $\partial\Omega$ can be covered by finitely many such compact faces, there exists a constant $\kappa_{\max}>0$ such that \eqref{eq:boundedCurvature} holds with $\kappa_{\max}$, for almost every $p \in \partial\Omega$. By \eqref{eq:GM} and \eqref{eq:boundedCurvature}, for every $\eps > 0$ we obtain
\begin{align}\label{eq:nunMalLangsam}
\begin{split}
 \sa [u, u] & \geq \sum_{j = 1}^d \int_\Omega |\nabla u_j|^2 - \kappa_{\max} \int_{\partial \Omega} |u|^2 \\
 & \geq \left (1 - \kappa_{\max} \eps \right) \sum_{j = 1}^d \int_\Omega |\nabla u_j|^2 - \kappa_{\max} \beta (\eps) \int_\Omega |u|^2
\end{split}
\end{align}
for all $u \in \dom \sa$. If we fix $\eps$ such that $\kappa_{\max} \eps < 1$ and choose $\mu = - \kappa_{\max} \beta (\eps)$, \eqref{eq:semibounded} follows. In particular, \eqref{eq:anorm} defines a norm on $\dom \sa$. Furthermore, it follows from \eqref{eq:nunMalLangsam} and the continuity of the trace operator $H^1 (\Omega) \to L^2 (\partial \Omega)$ that the norm $\|\cdot\|_\sa$ is equivalent to the usual norm of $H^1 (\Omega)^d$ on $\dom \sa$. Since $\dom \sa = V$ is a closed subspace of $H^1 (\Omega)^d$, it follows that $V$ is complete with respect to $\|\cdot\|_\sa$. Hence it follows that the operator $A$ is well-defined and self-adjoint with spectrum bounded from below.

Finally, since $\dom \sa \subset H^1 (\Omega)^d$ is compactly embedded into $L^2 (\Omega)^d$, it follows that $A$ has a compact resolvent and, hence, its spectrum is purely discrete.
\end{proof}

It is a consequence of the structure of the form $\sa$ that, in the interior of $\Omega$, $A$ acts componentwise as the Laplacian, as we now show. It will also turn out that if $\Omega$ is a lip domain, then all eigenvalues of $A$ are positive, something we will show in the course of the proof of Theorem~\ref{thm:hot-spots-III}.

\begin{lemma}\label{lem:A-Laplacian}
Under the assumptions of Proposition~\ref{prop:A-general}, if $u = (u_1,\ldots, u_d)^\top \in \dom A$ and $\lambda \in \R$ satisfy $A u = \lambda u$, then, for all $j=1,\ldots,d$, $u_j$ is analytic and $-\Delta u_j = \lambda u_j$ in $\Omega$. In particular, an orthonormal basis of $L^2 (\Omega)^d$ consisting of real-valued eigenfunctions of $A$ exists.
\end{lemma}

\begin{proof}
Since, as noted earlier, $C_0^\infty(\Omega)^d \subset V$, by assumption on $u$
\begin{displaymath}
	\sum_{j=1}^d \int_\Omega \langle \nabla u_j, \nabla v_j \rangle = \sa [u,v] =  \lambda \sum_{j=1}^d \int_\Omega \langle u, v\rangle
\end{displaymath}
for all $v = (v_1,\ldots,v_d)^\top \in C_0^\infty (\Omega)^d$. Fixing $j=1,\ldots,d$ and choosing $v_j = w \in C_0^\infty(\Omega)$ arbitrary, $v_i=0$ for $i \neq j$, we see immediately that $u_j$ thus satisfies
\begin{displaymath}
	\int_\Omega \langle \nabla u_j, \nabla w\rangle = \lambda \int_\Omega u_j \overline{w}
\end{displaymath}
for all $w \in C_0^\infty (\Omega)$. A completely standard argument now shows that $-\Delta u_j = \lambda u_j$ in $L^2(\Omega)$ and hence, by equally standard regularity theory, $u_j$ is analytic and the eigenvalue equation is satisfied everywhere in $\Omega$.

Moreover, since the actual eigenvalue equation is the Helmoltz equation inside $\Omega$, the real and imaginary parts of eigenfunctions of $A$ are eigenfunctions themselves and, hence, an orthonormal basis of real-valued eigenfunctions can be found.
\end{proof}

We next investigate the domain and action of the operator $A$ more closely. Recall that space $V = \dom \sa$ is defined in \eqref{eq:H}.

\begin{lemma}
\label{lem:Weingarten}
Let $\Omega \subset \R^d$ be piecewise smooth in the sense of Definition~\ref{def:piecewise-smooth}; if $\Gamma_{\rm D} \neq \emptyset$, then we assume that $\Gamma_{\rm D}$ is the union of one or several hyperplanar faces. Finally, we assume that $u = (u_1, \dots, u_d)^\top \in V$ has the following properties:
\begin{enumerate}
 \item[(a)] $u \in C^\infty (\Omega)^d$ and for each open set $U \subset \Omega$ whose closure only contains regular points of $\partial \Omega$, $u \in C^\infty (\overline U)^d$;
 \item[(b)] almost everywhere on $\partial \Omega$, the vector $\{\partial_l u_k - \partial_k u_l\}_{k, l} \nu$ is trivial;
 \item[(c)] $\langle \partial_\nu u, \nu\rangle = 0$ almost everywhere on $\Gamma_{\rm D}$;
 \item[(d)] $\Delta u_j \in L^2 (\Omega)$, $j = 1, \dots, d$.
\end{enumerate}
Then $u \in \dom A$ and $(A u)_j = - \Delta u_j$, for all $j = 1, \dots, d$.
\end{lemma}

\begin{proof}
Let $u$ satisfy the assumptions of the theorem. According to the definition of $A$ we have to show
\begin{align*}
 \sa [u, v] = - \sum_{j = 1}^d \int_\Omega \Delta u_j \overline{v_j}, \quad v \in V.
\end{align*}
Indeed, by Green's first identity, for each $v \in H^1(\Omega)^d$ we have
\begin{align*}
 \sa [u, v] + \sum_{j = 1}^d \int_\Omega \Delta u_j \overline{v_j} = \sum_{j = 1}^d (\partial_\nu u_j |_{\partial \Omega}, v_j |_{\partial \Omega} )_{\partial \Omega} - \int_{\partial \Omega} \langle L u, v \rangle,
\end{align*}
where $(\cdot,\cdot)_{\partial\Omega}$ denotes the duality between $H^{- 1/2} (\partial \Omega)$ and $H^{1/2} (\partial \Omega)$, see, e.g., \cite[Lemma 4.3]{McL}. To show that the right-hand side of the previous equation equals zero if $v \in V$, assume first that $v \in C^\infty (\overline \Omega)^d$ is such that its trace $v |_{\partial \Omega}$ is supported compactly in the interior of a set $\gamma = f (U)$ for a $C^\infty$-diffeomorphism $f : \R^{d - 1} \supset U \to \R^d$ on an open set $U$. We do not assume at this stage that $v \in V$. Then
\begin{align}\label{eq:computation}
 \sum_{j = 1}^d (\partial_\nu u_j |_{\partial \Omega}, v_j |_{\partial \Omega} )_{\partial \Omega} & = \int_\gamma \left\langle \begin{pmatrix} \langle \nabla u_1, \nu \rangle \\ \vdots \\ \langle \nabla u_d, \nu \rangle \end{pmatrix}, v \right\rangle.
\end{align}
Set $\tau_j (p) = \partial_j f (p)$, so that $\tau_1 (p), \dots, \tau_{d - 1} (p)$ form a basis of the tangential space of $\partial \Omega$ at any point $p \in \gamma$. Let us write $v = \sum_{j = 1}^{d - 1} c_j \tau_j + c_d \nu$, where 
\begin{align*}
 c_j = \frac{\langle v, \tau_j \rangle}{|\tau_j|^2}, \quad j = 1, \dots, d - 1, \qquad \text{and} \qquad c_d = \langle v, \nu \rangle.
\end{align*}
Then we get, everywhere on $\gamma$,
\begin{align}\label{eq:wieSchoen}
\begin{split}
 \left\langle \begin{pmatrix} \langle \nabla u_1, \nu \rangle \\ \vdots \\ \langle \nabla u_d, \nu \rangle \end{pmatrix}, v \right\rangle & = \sum_{j = 1}^{d - 1} \overline{c_j} \left\langle \begin{pmatrix} \langle \nabla u_1, \nu \rangle \\ \vdots \\ \langle \nabla u_d, \nu \rangle \end{pmatrix}, \tau_j \right\rangle + \overline{c_d} \left\langle \begin{pmatrix} \langle \nabla u_1, \nu \rangle \\ \vdots \\ \langle \nabla u_d, \nu \rangle \end{pmatrix}, \nu \right\rangle \\
 & = \sum_{j = 1}^{d - 1} \overline{c_j} \left\langle \begin{pmatrix} \langle \partial_1 u, \tau_j \rangle \\ \vdots \\ \langle \partial_d u, \tau_j \rangle \end{pmatrix}, \nu \right\rangle + \overline{c_d} \left\langle \begin{pmatrix} \langle \nabla u_1, \nu \rangle \\ \vdots \\ \langle \nabla u_d, \nu \rangle \end{pmatrix}, \nu \right\rangle.
\end{split}
\end{align}
We compute the expressions in the terms corresponding to the tangential vectors. In fact, note that
\begin{align*}
 \langle \partial_j (u \circ f), \nu \rangle & = \sum_{l = 1}^d \langle \nabla u_l, \partial_j f \rangle \nu_l
\end{align*}
and
\begin{align*}
 \langle \partial_l u, \partial_j f \rangle - \langle \nabla u_l, \partial_j f \rangle & = \sum_{k = 1}^d \partial_l u_k \partial_j f_k - \sum_{k = 1}^d \partial_k u_l \partial_j f_k.
\end{align*}
Together with assumption (b) this yields
\begin{align*}
 \left\langle \begin{pmatrix} \langle \partial_1 u, \tau_j \rangle \\ \vdots \\ \langle \partial_d u, \tau_j \rangle \end{pmatrix}, \nu \right\rangle & = \sum_{l = 1}^d \langle \nabla u_l, \partial_j f \rangle \nu_l + \sum_{l = 1}^d \sum_{k = 1}^d (\partial_l u_k - \partial_k u_l) \partial_j f_k \nu_l \\
 & = \left\langle \partial_j (u \circ f), \nu \right\rangle = \partial_j \langle u \circ f, \nu \rangle - \langle u \circ f, \partial_j \nu \rangle = \langle u, L \tau_j \rangle;
\end{align*}
from this and \eqref{eq:wieSchoen} it follows that
\begin{align*}
 \left\langle \begin{pmatrix} \langle \nabla u_1, \nu \rangle \\ \vdots \\ \langle \nabla u_d, \nu \rangle \end{pmatrix}, v \right\rangle & = \sum_{j = 1}^{d - 1} \overline{c_j} \langle u, L \tau_j \rangle + \overline{c_d} \left\langle \begin{pmatrix} \langle \nabla u_1, \nu \rangle \\ \vdots \\ \langle \nabla u_d, \nu \rangle \end{pmatrix}, \nu \right\rangle.
\end{align*}
Hence, \eqref{eq:computation} can be written
\begin{align*}
 \sum_{j = 1}^d (\partial_\nu u_j |_{\partial \Omega}, v_j |_{\partial \Omega} )_{\partial \Omega} & = \sum_{j = 1}^{d - 1} \int_\gamma \frac{\langle \tau_j, v\rangle}{|\tau_j|^2} \langle u, L \tau_j \rangle + \int_\gamma \langle \nu, v \rangle  \left\langle \begin{pmatrix} \langle \nabla u_1, \nu \rangle \\ \vdots \\ \langle \nabla u_d, \nu \rangle \end{pmatrix}, \nu \right\rangle.
\end{align*}
By taking linear combinations, for all $v \in C^\infty (\overline \Omega)^d$ whose support only contains regular points of $\partial \Omega$ we obtain
\begin{align}\label{eq:almostFinal}
 \sum_{j = 1}^d (\partial_\nu u_j |_{\partial \Omega}, v_j |_{\partial \Omega} )_{\partial \Omega} & = \sum_{l = 1}^d  \left( \sum_{j = 1}^{d - 1} \int_{\gamma_l} \frac{\langle \tau_j^l, v\rangle}{|\tau_j^l|^2} \langle u, L \tau_j \rangle + \left( \partial_\nu u, \langle v, \nu^l \rangle \nu^l \right)_{\gamma_l} \right)
\end{align}
if $\gamma_1, \dots, \gamma_m$ are the smooth faces of $\partial\Omega$ such that $\partial \Omega = \cup_{l = 1}^N \overline{\gamma_l}$, $\tau_1^l, \dots, \tau_{d-1}^l$ are the tangent vectors of $\gamma_l$ associated with a regular parametrization, and $\nu^l$ denotes the exterior unit normal vector field of $\gamma_l$; moreover $(\cdot, \cdot)_{\gamma_l}$ denotes the duality between $H^{-1/2} (\gamma_l)^d$ and $H_0^{1/2} (\gamma_l)^d$ ($= H^{1/2} (\gamma_l)^d$, cf.\ \cite[Theorem 3.40]{McL}).

Now let $v \in V$ be arbitrary. First observe that if $\Gamma_{\rm D} \neq \emptyset$, so that, say, $\Gamma_{\rm D} = \gamma_1 \cup \dots \cup \gamma_m$, then by assumption on $V$ we have $\langle \tau_j^1,v\rangle = \dots = \langle \tau_j^m,v\rangle = 0$ for all $j=1,\ldots,d-1$, while assumption (c) of the lemma says exactly that $\langle \partial_\nu u, \nu\rangle = 0$ almost everywhere on $\Gamma_{\rm D}$, meaning that the integral over $\Gamma_{\rm D} = \gamma_1 \cup \dots \cup \gamma_m$ in \eqref{eq:almostFinal} will be identically zero. We may thus assume that in fact $\Gamma_{\rm D} = \emptyset$.

Now by \cite[Chapter 8, Theorem 6.8]{EE} and \cite[Corollary 5.1.15]{AH} there exists a sequence $(v^n)_n$ in $C^\infty (\overline \Omega)^d$ such that the support of each $v^n$ only contains regular points and $v^n \to v$ in $H^1 (\Omega)^d$. (This is possible as the set on which $\partial \Omega$ is non-smooth is a set of zero $d - 1$-dimensional Hausdorff measure in $\R^d$ and by a capacity argument the functions in $C_0^\infty (\R^d)$ whose supports do not contain the nonsmooth parts of $\partial \Omega$ are dense in $H^1 (\R^d)$.) In particular, $v^n |_{\gamma_l}$ converges to $v |_{\gamma_l}$ in $H_0^{1/2} (\gamma_l)$ for each $l$ and, since $\nu^l$ is smooth and bounded on $\gamma_l$, even 
\begin{align*}
 \langle v^n, \nu^l \rangle \nu_j^l \to \langle v, \nu^l \rangle \nu_j^l = 0
\end{align*}
in $H_0^{1/2} (\gamma_l)$ holds for each $l$ and each $j$. Hence,
\begin{align*}
 \sum_{j = 1}^d (\partial_\nu u_j |_{\partial \Omega}, v_j |_{\partial \Omega} )_{\partial \Omega} & = \int_{\partial \Omega} \langle u, L v \rangle.
\end{align*}
Now the statement follows from the symmetry of the shape operator.
\end{proof}

The previous lemma can be applied to show that, as claimed, a part of the spectrum of $A$ consists of the eigenvalues of the Laplacian $- \Delta_{{\rm ND}}$ on $L^2(\Omega)$ with Neumann boundary conditions on $\Gamma_{\rm N}$ and Dirichlet conditions on $\Gamma_{\rm D}$, which we denote by $\mu_n$. We fix an orthonormal basis in $L^2(\Omega)$ of eigenfunctions $\psi_n \in \dom (- \Delta_{{\rm ND}})$, that is, such that
\begin{displaymath}
	- \Delta_{{\rm ND}} \psi_n = \mu_n \psi_n, \qquad n \in \N.
\end{displaymath}
We assume without loss of generality that the $\psi_n$ are real-valued.

\begin{proposition}
\label{prop:a-decomposition}
Let $\Omega \subset \R^d$ be a bounded Lipschitz domain satisfying Hypothesis~\ref{hyp:dom}; if $\Gamma_{\rm D} \neq \emptyset$, then we assume it consists of one or several hyperplanar faces. If $- \Delta_{{\rm ND}} \psi = \mu \psi$ for some $\mu > 0$ and $\psi \neq 0$, then $u = \nabla \psi$ belongs to $\dom A$, and $A u = \mu u$. Furthermore, the functions $\frac{1}{\sqrt{\mu_n}} \nabla \psi_n$, $\mu_n \neq 0$, form an orthonormal basis of the closed subspace
\begin{displaymath}
	\nabla H^1_{0,\Gamma_{\rm D}} (\Omega) := \{ \varphi \in H^1(\Omega): \varphi|_{\Gamma_{\rm D}} = 0\}
\end{displaymath}
of $L^2 (\Omega)^d$.
\end{proposition}

\begin{proof}
Let $\psi \in \dom (- \Delta_{{\rm ND}})$ satisfy $- \Delta_{{\rm ND}} \psi = \mu \psi$. Under the assumptions on $\Omega$ imposed in the proposition, $\psi$ belongs to $H^2 (\Omega)$ and, hence, $u := \nabla \psi \in H^1 (\Omega)^d$. Moreover, by elliptic regularity, $u$ satisfies condition (a) of Lemma \ref{lem:Weingarten}. For condition (b) it suffices to note that
\begin{align*}
 \partial_l u_k - \partial_k u_l = \partial_l \partial_k \psi - \partial_k \partial_l \psi = 0, \quad k, l \in \{1, \dots, d\},
\end{align*}
identically in $\Omega$.

Next, we show that $\langle \partial_\nu u, \nu \rangle = 0$ almost everywhere on $\Gamma_{\rm D}$. Fix any regular $p \in \Gamma_{\rm D}$ and let $N$ be the constant outer unit normal vector of the face containing $p$. Denote by $T^1, \dots, T^{d-1}$ an orthonormal basis of the tangential space of that face. Then
\begin{align*}
 \partial_N u_j (p) = \langle N, \nabla \partial_j \psi (p)\rangle = \partial_j \langle N, \nabla \psi \rangle
\end{align*}
holds for $j = 1, \dots, d$. Hence,
\begin{align*}
 \langle \partial_\nu u (p), \nu\rangle = \langle \nabla \langle N, \nabla \psi (p) \rangle, N \rangle = \partial_N \partial_N \psi (p) = \Delta \psi (p) - \sum_{j = 1}^{d-1} \partial_{T^j} \partial_{T^j} \psi (p) = 0,
\end{align*}
where we have employed the rotational invariance of the Laplacian to write it relative to the coordinate system determined by $N,T^1,\ldots,T^{d-1}$, the fact that $\psi$ is smooth up to the boundary locally around $p$ and satisfies $\Delta \psi (p) = - \mu \psi (p) = 0$, and the fact that $\partial_{T^j} \partial_{T^j} \psi (p) = 0$ for each $j$ since $\psi=0$ on $\Gamma_{\rm D}$, which by assumption is locally a hyperplane near $p$. Thus condition (c) of Lemma \ref{lem:Weingarten} holds.

Finally,
\begin{align*}
 \Delta u_j = \Delta \partial_j \psi = - \mu \partial_j \psi = - \mu u_j, \quad j = 1, \dots, d,
\end{align*}
which assures condition (d) of Lemma \ref{lem:Weingarten}. Hence it follows from the lemma that $u \in \dom A$ and
\begin{align*}
 (A u)_j = - \Delta u_j = \mu u_j, \quad j = 1, \dots, d,
\end{align*}
that is, $u \in \ker (A - \mu)$. 

Next, note that the fields $\frac{1}{\sqrt{\mu_n}} \nabla \psi_n$, $\mu_n \neq 0$, are mutually orthonormal, since
\begin{align*}
 \int_\Omega \langle \nabla \psi_n, \nabla \psi_m \rangle = - \int_\Omega \psi_m \Delta \psi_n = \mu_n \int_\Omega \psi_n \psi_m, \quad n, m = 1, 2, 3, \dots.
\end{align*}
Now assume that $\nabla \psi$ is orthogonal to all these fields for some $\psi \in H^1_{0,\Gamma_{\rm D}} (\Omega)$. Then
\begin{align*}
 0 & = \int_\Omega \langle \nabla \psi_n, \nabla \psi \rangle = - \int_\Omega \psi \Delta \psi_n = \mu_n \int_\Omega \psi \psi_n, \quad n = 1, 2, 3, \dots,
\end{align*}
which yields that $\psi$ is orthogonal to all eigenfunctions of $- \Delta_{{\rm ND}}$, except for the constant function $\psi_1$ in the case $\Gamma_{\rm D} = \emptyset$. Thus $\psi$ is constant and, hence, $\nabla \psi = 0$ identically in $\Omega$.
\end{proof}

Proposition~\ref{prop:a-decomposition} implies, in particular, that $A$ decomposes orthogonally according to a type of Helmholtz decomposition; in fact, this decomposition holds under conditions of minimal regularity:

\begin{lemma}
Let $\Omega \subset \R^d$ be a bounded Lipschitz domain and let $\Gamma_{\rm N}, \Gamma_{\rm D} \subset \partial \Omega$ be closed, connected subsets of $\partial\Omega$ whose intersection has zero $d-1$-dimensional Hausdorff measure, and such that $\Gamma_{\rm N} \cup \Gamma_{\rm D} = \partial \Omega$. Then we have the direct sum decomposition
\begin{equation}
\label{eq:HelmholtzNew}
 L^2 (\Omega) = \nabla H_{0, \Gamma_{\rm D}}^1 (\Omega) \oplus \left\{u \in L^2 (\Omega)^d : \diver u = 0, \langle u |_{\Gamma_{\rm N}}, \nu \rangle = 0 \right\}.
\end{equation}
\end{lemma}

Note that $\Gamma_{\rm D} = \emptyset$ is allowed in the above lemma, as is $\Gamma_{\rm N} = \emptyset$.

\begin{proof}
The Poincar\'e inequality on $H_{0, \Gamma_{\rm D}}^1 (\Omega)$ implies that $\nabla H_{0, \Gamma_{\rm D}}^1 (\Omega)$ is complete and, hence, closed in $L^2 (\Omega)^d$. Morever, if $u$ is orthogonal to $\nabla H_{0, \Gamma_{\rm D}}^1 (\Omega)$ in $L^2 (\Omega)$, then for each $\phi \in H_{0, \Gamma_{\rm D}}^1 (\Omega)$ we have
\begin{align*}
 0 = \int_\Omega \langle u, \nabla \phi \rangle = - \int_\Omega \phi \diver u + \left( \langle u, \nu \rangle, \phi \right)_{\partial \Omega}. 
\end{align*}
On the one hand, using $\phi \in C_0^\infty (\Omega)$ this implies $\diver u = 0$ identically. On the other hand, inserting an arbitrary $\phi \in H_{0, \Gamma_{\rm D}}^1 (\Omega)$ in the above relation, we obtain $\langle u |_{\Gamma_{\rm N}}, \nu\rangle = 0$. Finally, it follows from the reverse computation that any $u \in L^2 (\Omega)^d$ such that $\diver u = 0$ and $\langle u |_{\Gamma_{\rm N}}, \nu\rangle = 0$ is orthogonal to $\nabla H_{0, \Gamma_{\rm D}}^1 (\Omega)$; this completes the proof.
\end{proof}

We will denote the projection of $A$ onto the second subspace in \eqref{eq:HelmholtzNew} by $A_2$, and its ordered eigenvalues by $\tau_k (A_2)$. In dimension two, and under sufficient hypotheses on $\Omega$, $A_2$ can be characterized relatively simply; the $\tau_k$ are the eigenvalues of the (scalar) Laplacian with Dirichlet conditions on $\Gamma_{\rm N}$ and Neumann conditions on $\Gamma_{\rm D}$ (see \cite[Theorem~3.5]{Rprep} for the case $\Gamma_{\rm D} = \emptyset$, or \cite[Proposition~3.4]{ARprep} for the general case). For $\Omega \subset \R^d$ the situation becomes much more complicated; we will discuss this further in Section~\ref{sec:curl-curl}. However, for the purposes of proving Theorem~\ref{thm:hot-spots-III}, we will not actually need any further properties of $A_2$.

\section{Proof of Theorem~\ref{thm:hot-spots-III}}
\label{sec:proof}

With the help of Proposition~\ref{prop:a-decomposition} we can build on the ideas set out in \cite{Rprep} for the case of two-dimensional lip domains. We first look more closely at lip domains in $\R^d$. We recall that a subspace $X$ of $\R^d$ is called a sublattice if $x \in X$ implies $|x| \in X$, where the modulus is taken componentwise.

\begin{lemma}
\label{lem:sublattice}
Let $d \geq 2$, $\nu$ a unit vector in $\R^d$ and $X = \R^d \ominus \spann \{\nu\}$. Then $X$ is a sublattice if and only if either $\nu$ has exactly two nonzero components and these two have opposite sign or $\nu = \pm \e_j$ for some standard basis vector $\e_j$.
\end{lemma}

Thus, if $\Omega$ is a lip domain, we have that $\R^d \ominus \spann \{\nu(x)\}$ is a sublattice of $\R^d$ for almost all $x \in \partial\Omega$. 

\begin{proof}
The proof is based on the following characterization of sublattices, see \cite[Lemma A.1]{KKVW09}: a subspace $X$ is a sublattice if and only if there exists a basis $\{u^1, u^2, \dots, u^m \}$ of $X$ (which we are going to call a sublattice basis) such that all entries of each $u^k$ are non-negative and for each $j$ there is at most one $k$ such that $u_j^k$ is nonzero. 

Let us assume that the space $X$ as defined in the lemma is a sublattice, and let $\{u^1, u^2, \dots, u^{d - 1} \}$ be such a sublattice basis. As each $u^k$ is non-trivial and $X$ is $d-1$-dimensional, none of them can have more than two nonzero entries, and if one of these vectors has two nonzero entries then the remaining $u^k$ must be positive multiples of standard basis vectors. In case there exists $k_0$ such that $u^{k_0}$ has exactly two nonzero elements, say $u^{k_0} = (a, b, 0, \dots, 0)^\top$ for some $a, b > 0$, then $u^{k_0} \perp \nu$ implies $a \nu_1 + b \nu_2 = 0$, that is, $\nu = \pm \frac{1}{a^2 + b^2} (b, -a, 0, \dots, 0)^\top$. On the other hand, if each $u^k$, $k = 1, \dots, d - 1$, has exactly one non-trivial entry then $u^k \perp \nu$ for all $k$ implies $\nu_j = 0$ for all but one $j$, that is, $\nu = \pm \e_j$ for some $j$. 

For the reverse implication, in the case that $\nu = \pm \e_j$ for some $j$, $\{\e_m : m \neq j\}$ is a sublattice basis of $X$. In the other case, where $\nu$ has exactly two nonzero components and these have opposite signs, say $\nu = (b, -a, 0, \dots, 0)^\top$ for some positive $a, b$, then 
\begin{align*}
 \left\{ \begin{pmatrix} a \\ b \\ 0 \\ \vdots \\ 0 \end{pmatrix}, \e_3, \dots, \e_d \right\}
\end{align*}
is a sublattice basis for $X$. In each case it follows that $X$ is a sublattice of $\R^d$.
\end{proof}

The following lemma will be needed in the proof of Theorem \ref{thm:hot-spots-III}.

\begin{lemma}\label{lem:normalComponents}
Assume that $\Omega \subset \R^d$ is a piecewise smooth lip domain. Then for each regular point $p_0 \in \partial \Omega$ there exist an open neighborhood $\cO (p_0)$ in $\partial \Omega$ and numbers $k, l \in \{1, \dots, d\}$ such that $\nu_j (p) = 0$ for all $p \in \cO (p_0)$ and each $j \neq k, l$.
\end{lemma}

\begin{proof}
Suppose the conclusion of the lemma is not satisfied at the regular point $p_0 \in \partial\Omega$. We will fix a sufficiently small neighborhood $\cV \subset \R^d$ of $p_0$, which we may take to be a $d$-dimensional open cube of the form $\cV = U \times I$ for a $d-1$-dimensional cube $U \subset \R^{d-1}$ and an interval $I \subset \R$. In what follows we will always consider the relative topology in the set $\partial\Omega \cap \mathcal{V}$.

We may assume that there exists a smooth function $F: U \to I$ such that
\begin{displaymath}
 	\cO (p_0) := \partial\Omega \cap \cV = \{ (x_1,\ldots,x_{d-1},F(x_1,\ldots,x_{d-1})) : (x_1,\ldots,x_{d-1}) \in U\}.
\end{displaymath}
We may also assume that the coordinate axes are chosen such that, within $\mathcal{V}$, $p_0$ is in the closure of the two sets
\begin{displaymath}
\begin{aligned}
	\cO_1 &:= \{ p \in \partial \Omega \cap \cV: \nu (p) = \tfrac{1}{1+a^2}(-a,0,\ldots,0,1)^\top \text{ for some } a = a(p) > 0\},\\
	\cO_2 &:= \{ p \in \partial \Omega \cap \cV: \nu (p) = \tfrac{1}{1+b^2}(0,-b,\ldots,0,1)^\top \text{ for some } b = b(p) > 0\};
\end{aligned}
\end{displaymath}
here the functions $a$ and $b$ will be smooth in $\cO_1$ and $\cO_2$ respectively (which is true since $F$ is smooth). Note that $\cO_1$ and $\cO_2$ are relatively open in $\partial\Omega \cap \mathcal{V}$. We also set
\begin{displaymath}
	\Upsilon := \{ p \in \partial \Omega \cap \cV: \nu (p) = (0,\ldots,0,1)^\top \}
\end{displaymath}
which, again since $p_0$ is regular, is a relatively closed set. Moreover, $p_0 \in \Upsilon \cap \partial\cO_1 \cap \partial\cO_2$. (See Figure~\ref{fig:level-surface-configuration}-left, but note that \emph{a priori} these three sets need not exhaust $\partial \Omega \cap \cV$, and that $\cO_1$ and $\cO_2$ need not be connected.)

Observe that the vector $(-\frac{\partial F}{\partial x_1}, \ldots, -\frac{\partial F}{\partial x_{d-1}}, 1)^\top$ is always an outer normal vector to $\partial\Omega$ in $\cV$. It follows from the definition of the respective sets that, for any $p=(p',p_d) \in \partial\Omega\cap\mathcal{V}$,
\begin{displaymath}
\begin{aligned}
	\frac{\partial F}{\partial x_1} (p') & \neq 0 \qquad &&\text{if } p \in \cO_1,\\
	\frac{\partial F}{\partial x_2} (p') & \neq 0 \qquad &&\text{if } p \in \cO_2,\\
	\nabla F(p') &= 0 \qquad &&\text{if } p \in \Upsilon,
\end{aligned}
\end{displaymath}
while $\frac{\partial F}{\partial x_2} (p'),\ldots,\frac{\partial F}{\partial x_{d-1}} (p') = 0$ whenever $(p', F (p')) \in \cO_1$ (and mutatis mutandis for $\cO_2$). It follows from the fact that $\nabla F(p') = 0$ if $(p', F (p')) \in \Upsilon$ that $F$, as a smooth function, is constant on each connected component of $\{p' \in U : (p', F (p)) \in \Upsilon\}$. Moreover, within each connected component of $\cO_1$, all level surfaces of $F$ are of the form $x_1 = \text{const}$, and, similarly, inside each connected component of $\cO_2$, all level surfaces of $F$ are of the form $x_2 = \text{const}$; in particular, within each connected component of $\cO_1$, $F$ is strictly monotonic in the direction $e_1$, and analogously for $\cO_2$ and $e_2$. (See Figure~\ref{fig:level-surface-configuration}-center.)
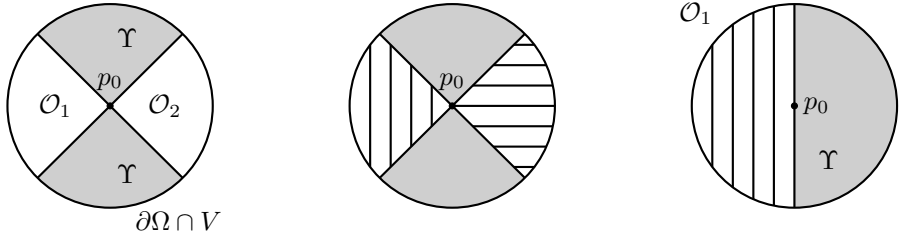
\begin{figure}[ht]
\begin{tikzpicture}[scale=0.9]
\filldraw[color=lightgray!75] (-sqrt{1.125},sqrt{1.125}) -- (0,0) -- (sqrt{1.125},sqrt{1.125}) arc[start angle=45, end angle=135, radius=1.5] -- cycle;
\filldraw[color=lightgray!75] (sqrt{1.125},-sqrt{1.125}) -- (0,0) -- (-sqrt{1.125},-sqrt{1.125}) arc[start angle=225, end angle=315, radius=1.5] -- cycle;
\draw[thick] (0,0) circle[radius=1.5];
\filldraw (0,0) circle[radius=1.25pt];
\node at (0,0.1) [anchor=south] {$p_0$};
\draw[thick] (-sqrt{1.125},sqrt{1.125}) -- (sqrt{1.125},-sqrt{1.125});
\draw[thick] (-sqrt{1.125},-sqrt{1.125}) -- (sqrt{1.125},sqrt{1.125});
\node at (-0.8,0) {$\cO_1$};
\node at (0.8,0) {$\cO_2$};
\node at (0.25,-1) {$\Upsilon$};
\node at (0.25,1) {$\Upsilon$};
\node at (1,-1.7) {$\partial\Omega \cap V$};
\filldraw[color=lightgray!75] (5-sqrt{1.125},sqrt{1.125}) -- (5,0) -- (5+sqrt{1.125},sqrt{1.125}) arc[start angle=45, end angle=135, radius=1.5] -- cycle;
\filldraw[color=lightgray!75] (5+sqrt{1.125},-sqrt{1.125}) -- (5,0) -- (5-sqrt{1.125},-sqrt{1.125}) arc[start angle=225, end angle=315, radius=1.5] -- cycle;
\draw[thick] (5,0) circle[radius=1.5];
\draw[thick] (5-sqrt{1.125},sqrt{1.125}) -- (5+sqrt{1.125},-sqrt{1.125});
\draw[thick] (5-sqrt{1.125},-sqrt{1.125}) -- (5+sqrt{1.125},sqrt{1.125});
\filldraw (5,0) circle[radius=1.25pt];
\node at (5,0.1) [anchor=south] {$p_0$};
\foreach \X in {0,0.3,0.6,0.9}
{
\draw[thick] (5-\X,-\X) -- (5-\X,\X);
}
\draw[thick] (5-1.2,sqrt{0.81}) -- (5-1.2,-sqrt{0.81});
\draw[thick] (5+0,0) -- (5+1.5,0);
\draw[thick] (5+0.3,0.3) -- (5+sqrt{2.16},0.3);
\draw[thick] (5+0.6,0.6) -- (5+sqrt{1.89},0.6);
\draw[thick] (5+0.9,0.9) -- (5+sqrt{1.44},0.9);
\draw[thick] (5+0.3,-0.3) -- (5+sqrt{2.16},-0.3);
\draw[thick] (5+0.6,-0.6) -- (5+sqrt{1.89},-0.6);
\draw[thick] (5+0.9,-0.9) -- (5+sqrt{1.44},-0.9);
\filldraw[color=lightgray!75] 
(10,-1.5) -- (10,1.5) arc[start angle=90, end angle=-90, radius=1.5] -- cycle;
\draw[thick] (10,0) circle[radius=1.5];
\filldraw (10,0) circle[radius=1.25pt];
\node at (10,0) [anchor=west] {$p_0$};
\draw[thick] (10,-1.5) -- (10,1.5);
\draw[thick] (10-0.3,-sqrt{2.16}) -- (10-0.3,+sqrt{2.16});
\draw[thick] (10-0.6,-sqrt{1.89}) -- (10-0.6,+sqrt{1.89});
\draw[thick] (10-0.9,-sqrt{1.44}) -- (10-0.9,+sqrt{1.44});
\draw[thick] (10-1.2,-sqrt{0.81}) -- (10-1.2,+sqrt{0.81});
\node at (10.5,-0.8) {$\Upsilon$};
\node at (10-sqrt{1.125},sqrt{1.125}) [anchor=south east] {$\cO_1$};
\end{tikzpicture}
\caption{Left: A schematic representation (when $d=3$) of the regions $\Upsilon$, where $F$ is constant (in light gray), $\cO_1$ and $\cO_2$. Center: A depiction of the necessary form of the level surfaces of $F$ in $\cO_1$ (vertical lines) and $\cO_2$ (horizontal lines). Right: The only possible configuration consistent with $\Upsilon$ meeting $\cO_1$ at $p_0$. The vertical lines again represent the level surfaces of $F$.}
\label{fig:level-surface-configuration}
\end{figure}

Let $\omega_1$ be any connected component of $\cO_1$ in $\mathcal{V}$, which is (relatively) open in $\partial\Omega \cap \mathcal{V}$, as noted above. Now $\partial\omega_1$ will not in general be connected; we will distinguish between two possibilities for its connected components $\gamma$. Namely, we will call $\gamma$ \emph{non-essential} if there exists a neighborhood of $\gamma$ entirely contained in $\overline{\omega_1}$, and \emph{essential} otherwise (in which case $\gamma$ also belongs to the boundary of $(\partial\Omega \cap \mathcal{V}) \setminus \overline{\omega_1}$).

Note that by continuity of $F$ any connected component of $\partial\omega_1$, essential or otherwise, will be necessarily contained in a connected component of $\Upsilon$, on which $F$ is constant, as noted above.

We claim that every essential connected component of $\partial\omega_1$ is in fact of the form $\{x_1 = \text{const}\}$ inside $\mathcal{V}$. To see this, suppose such a connected component $\gamma$ of $\partial\omega_1$ is \emph{not} of this form. Since it is essential it cannot be a singleton; and since it separates $\omega_1$ from a connected component of $\partial\Omega \setminus \overline{\omega_1}$ inside $\mathcal{V}$ it cannot be properly contained in a set of the form $\{x_1 = \text{const}\}$. Thus there must exist two distinct points $X^\ast = (x_1^\ast,\ldots,x_d^\ast),Y^\ast = (y_1^\ast,\ldots,y_d^\ast) \in \gamma$ for which $x_1^\ast \neq y_1^\ast$.

Now since $F$ is constant on $\gamma$, necessarily $$x_d^\ast = F(x_1^\ast,\ldots,x_{d-1}^\ast) = F(y_1^\ast,\ldots,y_{d-1}^\ast) = y_d^\ast.$$ On the other hand, $F$ is strictly monotonic in the direction $e_1$ within $\omega_1$ (concretely, supposing without loss of generality that it is increasing, this means that $x_d = F(x_1, x_2,\ldots, x_{d-1}) < F(y_1, y_2, \ldots, y_{d-1}) = y_d$ for all $x, y \in \omega_1$ whenever $x_1 < y_1$, independently of the other coordinates). Since $F$ is continuous on $\{p' \in U : (p', F (p')) \in \overline{\omega_1}\}$ and constant on sets of the form $\{x_1 = \text{const}\}$ within $\{p' \in U : (p', F (p')) \in \omega_1\}$, it follows that $F(x_1,\ldots,x_{d-1}) \neq F(y_1,\ldots,y_{d-1})$ on $\overline{\omega_1}$ whenever $x_1 \neq y_1$, which is a contradiction to $x_d^\ast = y_d^\ast$.

Hence the assumption that $\gamma$ was not of the form $\{x_1 = \text{const}\}$ inside $\mathcal{V}$ must be false. An analogous statement is true, with the same reasoning, of each connected component $\omega_2$ of $\cO_2$, that is, every essential connected component of $\partial\omega_2$ is of the form $\{x_2 = \text{const}\}$. Since $\cO_1$ and $\cO_2$ are the unions of their connected components, the same is also true of $\partial\cO_1$ and $\partial\cO_2$: every essential component of their boundaries must be of the form $\{x_1 = \text{const}\}$ or $\{x_2 = \text{const}\}$, respectively, inside $\mathcal{V}$.

We claim that this is a contradiction to the assumption that $p_0 \in \partial\cO_1 \cap \partial\cO_2 \subset \Upsilon$. Indeed, the fact that $p_0$ belongs to the boundary of both means that it must belong to an essential connected component of the boundaries of both $\partial\cO_1$ and $\partial\cO_2$ (cf.\ Figure~\ref{fig:level-surface-configuration}-right), which is impossible, since they are, respectively, of the form $\{x_1 = \text{const}\}$ and $\{x_2 = \text{const}\}$, but $\cO_1 \cap \cO_2 = \emptyset$.
\end{proof}

We are now in a position to prove Theorem~\ref{thm:hot-spots-III}, which we recall also contains Theorem~\ref{thm:hot-spots-I} as a special case, corresponding to $\Gamma_{\rm D} = \emptyset$.

\begin{proof}[Proof of Theorem~\ref{thm:hot-spots-III}]
Denote by $\eta_1 = \eta_1 (A)$ the first eigenvalue of $A$ and let $u \in \ker (A - \eta_1)$; without loss of generality we may assume that $u$ is real-valued, see Lemma \ref{lem:A-Laplacian}. Then $u$ minimizes the Rayleigh quotient associated with $\sa$,
\begin{displaymath}
	u = \argmin \left\{ \frac{\sum_{j = 1}^d \int_\Omega |\nabla w_j|^2 - \int_{\partial \Omega} \langle L w, w \rangle}{\int_\Omega |w|^2} : 0 \neq w \in V \right\}.
\end{displaymath}

\emph{Step 1:} Under our assumptions on $\Omega$, we have that $|u| \in V$ and $|u|$ is also an eigenfield associated with $\eta_1$. Indeed, our assumptions on $\Omega$ together with Lemma~\ref{lem:sublattice} imply that
\begin{displaymath}
	|u| := \left(\begin{matrix} |u_1|\\ \vdots\\ |u_d| \end{matrix}\right) \in H^1 (\Omega)^d
\end{displaymath}
also satisfies the boundary condition $\langle |u|, \nu \rangle = 0$ at almost every point of $\Gamma_{\rm N}$. Furthermore, as $\Gamma_{\rm D}$ is perpendicular to a coordinate axis, i.e.\ its outer unit normal $\nu$ equals $\pm \e_j$ for some $j$, since $u$ is normal to $\Gamma_{\rm D}$, also $|u|$ is. Thus $|u| \in V$.

Next we consider the Rayleigh quotient of $|u|$. It is clear that $|\nabla |u_j|| \leq |\nabla u_j|$ holds for $j = 1, \dots, d$. For the boundary integral, if $p_0 \in \Gamma_{\rm N}$ is a regular point then by Lemma \ref{lem:normalComponents} there exists an open neighborhood $\cO (p_0)$ of $p_0$ in $\partial \Omega$ such that, without loss of generality,
\begin{align}\label{eq:Normal}
 \nu (p) = \left(0, \dots, 0, \beta (p), \alpha (p)\right)^\top
\end{align}
holds for all $p \in \cO (p_0)$, where $\alpha (p)$ and $\beta (p)$ have opposite sign, say $\beta (p) \leq 0 < \alpha (p)$, and $\alpha (p)^2 + \beta (p)^2 = 1$. In particular, locally around $p_0$, $\partial \Omega$ is flat in the directions of $e_1, \dots, e_{d-2}$. More precisely, if $\cO (p_0)$ is the graph of a smooth function $F : U \to \R$ for an open set $U \subset \R^{d-1}$, and we choose the parametrization 
\begin{align*}
 f : U \to \R^d, \quad f (x_1, \dots, x_{d-1}) = \big( x_1, \dots, x_{d-1}, F (x_1, \dots, x_{d-1}) \big)^\top,
\end{align*}
then 
\begin{align}\label{eq:tangent}
 \partial_j f = \ee_j + \partial_j F \ee_d, \quad j = 1, \dots, d-1,
\end{align}
and, possibly up to a sign,
\begin{align*}
 \nu = \sqrt{1 + (\partial_1 F)^2 + \dots + (\partial_{d-1} F)^2} \left(- \partial_1 F, \dots, - \partial_{d-1} F, 1 \right)^\top;
\end{align*}
a comparison of the latter with \eqref{eq:Normal} gives
\begin{align*}
 \partial_1 F = \dots = \partial_{d - 2} F = 0
\end{align*}
constantly in $\Omega$. In particular, $F$ and, hence, $\nu$ are independent of the variables $x_1, \dots, x_{d-2}$. Therefore $\partial_j \nu (x) = 0$ for $j = 1, \dots, d - 2$ and all $x \in U$, and it follows from the definition of $L_{p_0}$ that $L_{p_0}$ has at most one non-trivial eigenvalue $\kappa_{d - 1} (p_0)$. Since $\nu (p_0), \ee_1, \dots, \ee_{d-2}$ belong to $\ker L_{p_0}$ (cf.\ \eqref{eq:tangent}), a normalized eigenvector corresponding to $\kappa_{d-1} (p_0)$ is
\begin{align*}
 w = \left( 0, \dots, 0, - \alpha (p_0), \beta (p_0) \right)^\top.
\end{align*}
Furthermore, $\langle u (p), \nu (p) \rangle = 0$ implies that $u$ is of the form
\begin{align*}
 u (p) = \gamma (p) \left( *, \dots, *, - \alpha (p), \beta (p) \right)^\top.
\end{align*}
for some $\gamma (p) \in \R$. Hence the spectral theorem for the symmetric matrix $L_{p_0}$ gives
\begin{align*}
 \left\langle L_{p_0} |u (p_0)|, |u (p_0)| \right\rangle & = \kappa_{d - 1} (p_0) \langle |u (p_0)|, w \rangle^2 = \kappa_{d - 1} (p_0) \gamma (p)^2  \\ 
 & = \kappa_{d - 1} (p_0) \langle u (p_0), w \rangle^2 = \left\langle L_{p_0} u (p_0), u (p_0) \right\rangle,
\end{align*}
where we have used that both $- \beta (p_0)$ and $\alpha (p_0)$ are non-negative. If $\Gamma_{\rm D} \neq \emptyset$ and $p_0 \in \Gamma_{\rm D}$ is a regular point, then we automatically have $L_{p_0} = 0$, and so the same conclusion holds.

Summing up, we have shown that
\begin{displaymath}
	\frac{ \sa [|u|] }{\int_\Omega |u|^2} \leq \frac{ \sa [u] }{\int_\Omega |u|^2} = \eta_1,
\end{displaymath}
and thus $|u| = (|u_1|,\ldots,|u_d|)$ is also an eigenfunction for $A$ corresponding to $\eta_1$; in particular, by Lemma~\ref{lem:A-Laplacian}, for each $j$, $|u_j|$ is analytic and $-\Delta |u_j| = \eta_1 |u_j|$ pointwise in $\Omega$.

\emph{Step 2:} No component $u_j$ of the original eigenfunction $u$ can change sign in $\Omega$. Indeed, otherwise, $u_j$ must have at least two nodal domains, and thus a nodal set $\mathcal{N} \subset \Omega$ of positive $(d-1)$-dimensional measure separating them. Choose some open set $U \subset \Omega$ such that $U \cap \cN$ is a non-empty hypersurface. Then $|u_j|$ vanishes on $U \cap \cN$ by definition, but in particular $|u_j|$ attains a global minimum in $\Omega$ at every point in $U \cap \mathcal{N}$; thus, since $|u_j|$ is a smooth function, we also have $\nabla |u_j| = 0$ identically on $U \cap \mathcal{N}$. Thus both $|u_j|$ and its derivative on $U \cap \cN$ in the direction normal to that hypersurface vanish. As $- \Delta |u_j| = \eta_1 |u_j|$, a standard unique continuation argument implies that $|u_j|$, and thus $u_j$, is identically $0$ in $\Omega$, a contradiction to the assumption that $u_j$ changed sign.

\emph{Step 3:} We have $\eta_1 = \lambda < \tau_1 (A_2)$, and in particular, $\eta_1 > 0$. To see this, assume for a contradiction that some (without loss of generality real-valued) $u \in \ker (A - \eta_1)$ belongs to $\ker (A_2 - \eta_1)$ and, in particular, $\diver u = 0$ in $\Omega$.

If $\Gamma_{\rm D} \neq \emptyset$, we may assume without loss of generality that $\nu = -e_d$ almost everywhere on $\Gamma_{\rm D}$, so that $\Gamma_{\rm D}$ is a subset of the hyperplane $x_d = c$ for some $c \in \R$. Then for each $a \in \C^d$ we have
\begin{align}\label{eq:jetztSchlaegtsDreizehn}
\begin{split}
 \int_\Omega \langle a, u\rangle & = \int_\Omega \langle \nabla (\langle a, x\rangle - c), u \rangle = - \int_\Omega (\langle a, x\rangle - c) \diver u + \left( \langle a, x \rangle - c, \langle u, \nu\rangle \right)_{\partial \Omega}.
\end{split}
\end{align}
Plugging in $a = e_d$, the boundary term vanishes since $\langle u|_{\partial \Omega}, \nu \rangle = 0$ on $\Gamma_{\rm N}$ and $x_d = c$ on $\Gamma_{\rm D}$, and we obtain $\int_\Omega u_d = 0$. However, by Step 2, $u_d$ does not change sign in $\Omega$, and hence $u_d = 0$ identically. In particular, $\langle u|_{\partial \Omega}, \nu\rangle$ vanishes on $\partial \Omega$ regardless of whether $\Gamma_{\rm D} = \emptyset$ or not. Hence, \eqref{eq:jetztSchlaegtsDreizehn} reads
\begin{align*}
 \int_\Omega \langle a, u \rangle = 0
\end{align*}
for all $a \in \C^d$. Choosing $a = e_j$ for arbitrary $j$ and using again that none of the $u_j$ changes sign it finally follows that $u = 0$ identically in $\Omega$. Hence $\tau_1 (A_2) > \eta_1$, which also implies $\eta_1 = \lambda > 0$ by assumption on $\lambda$.

\emph{Step 4:} Conclusion. Let $\psi$ be an arbitrary eigenfunction of the Laplacian with Neumann boundary conditions on $\Gamma_{\rm N}$ and Dirichlet boundary conditions on $\Gamma_{\rm D}$ corresponding to $\lambda = \eta_1$, which we may assume to be real valued. By Proposition~\ref{prop:a-decomposition}, the function
\begin{displaymath}
 u = \left(\begin{matrix} u_1\\ \vdots \\ u_d \end{matrix}\right) := \nabla \psi \in \dom \sa = V
\end{displaymath}
belongs to $\ker (A - \eta_1)$. Hence, by Step 2, the components of $u$ do not change sign. Moreover, for each $j \in \{1, \dots, d\}$ we have $- \Delta u_j = \lambda u_j$, and since $\lambda > 0$, the maximum principle for subharmonic functions implies that either $u_j = \partial_j \psi > 0$ in $\Omega$, $u_j  = \partial_j \psi < 0$ in $\Omega$, or $u_j = \partial_j \psi = 0$ identically in $\Omega$. In particular, since $\psi$ is not constant, it follows that $\psi$ attains its minimum and maximum only on the boundary.
\end{proof}

\section{Hot spots on symmetric domains: proof of Theorem~\ref{thm:hot-spots-II}}
\label{sec:symm}

In this section we assume throughout that $\Omega$, in addition to Hypothesis~\ref{hyp:dom}, has a reflection symmetry in every hyperplane of the form $\{x_j = 0\}$, $j=1,\ldots,d$, and that there exists an orthant $\mathcal{O}$ such that $\Omega \cap \mathcal{O}$ is a lip domain (Definition~\ref{def:lip}). Denote by $-\Delta_{\rm N}$ the Neumann Laplacian on $L^2(\Omega)$, and note that Hypothesis~\ref{hyp:dom} guarantees that $\dom (-\Delta_{\rm N}) \subset H^2(\Omega)$. The following is standard.

\begin{lemma}
\label{lem:symm-basis}
There exists an orthonormal basis $(\psi_n)_n$ of $L^2(\Omega)$ consisting of eigenfunctions of $-\Delta_{\rm N}$ such that, for each $n \in \N$ and each $j=1,\ldots,d$, $\psi_n$ is either reflection symmetric or reflection antisymmetric in the plane $\{x_j = 0\}$, that is, either
\begin{displaymath}
\begin{aligned}
	\psi_n(x_1,\ldots,x_j,\ldots,x_d) &= \psi_n(x_1,\ldots,-x_j,\ldots,x_d) \quad &&\text{in}~\Omega,~\text{or}\\ 
	\psi_n(x_1,\ldots,x_j,\ldots,x_d) &= -\psi_n(x_1,\ldots,-x_j,\ldots,x_d) \quad &&\text{in}~\Omega,
\end{aligned}
\end{displaymath}
respectively.
\end{lemma}

In the antisymmetric case, $\psi_n$ satisfies a Dirichlet condition on $\{x_j=0\} \cap \overline{\Omega}$. In this case we will speak of a {\em $j$-antisymmetric function}. The next lemma is also standard.

\begin{lemma}
\label{lem:antisymm-mixed}
Fix $j=1,\ldots,d$ and denote by $(\psi_{n,j})_n$ the collection of eigenfunctions of $-\Delta_{\rm N}$ from Lemma~\ref{lem:symm-basis} which are $j$-antisymmetric, with associated eigenvalues $\mu_{n,j}$. Denote by $-\Delta_j^+$ the Laplacian on $\Omega_j^+ := \Omega \cap \{ x_j > 0 \}$ with Neumann conditions on $\partial\Omega \cap \{x_j > 0\}$ and Dirichlet conditions on $\{ x_j = 0\}$. Then $(\mu_{n,j})_n$ forms the totality of the spectrum of $-\Delta_j^{+}$, and $(\psi_{n,j}|_{\{x_j > 0\}})_n$ forms an orthonormal basis of $L^2 (\Omega_j^+)$ of eigenfunctions of $-\Delta_j^{+}$.

A corresponding statement holds for $\Omega_j^- := \Omega \cap \{ x_j < 0\}$ and $-\Delta_j^-$; in particular, $\mu_{1,j}$ is the first eigenvalue of both $\Omega_j^+$ and $\Omega_j^-$.
\end{lemma}

From these observations we will derive the following lemma.

\begin{lemma}
\label{lem:symm-tech}
Keeping the notation and assumptions of this section, let $\mu^j > 0$ be the smallest eigenvalue of $-\Delta_{\rm N}$ which has an eigenfunction $\psi^j$ which is antisymmetric in the hyperplane $\{x_j = 0\}$, $j=1,\ldots,d$. Then
\begin{enumerate}
\item[(a)] $\mu^j = \mu_{1,j}$, the smallest eigenvalue of $-\Delta_j^{\pm}$ on $\Omega_j^\pm$;
\item[(b)] for each $j=1,\ldots,d$, there exists a unique (up to scalar multiples) $j$-antisymmetric eigenfunction $\psi^j$ associated with $\mu^j$; moreover, $\{ x \in \Omega: \psi^j(x) = 0\}$ is exactly $\{ x_j = 0 \}$.
\end{enumerate}
\end{lemma}

Note that $\mu^j$ may not be simple (and there be more than one $j$ such that $\mu_2 = \mu^j$); we are merely affirming that $\mu^j$, as an eigenvalue of $-\Delta_{\rm N}$, can only have one associated eigenfunction (up to scalar multiples) which vanishes on $\{x_j = 0\}$. Likewise, $\mu_2$ may be multiple, but in the basis described in Lemma~\ref{lem:symm-basis}, all eigenfunctions associated with $\mu_2$ are smallest $j$-antisymmetric eigenfunctions for some $j$.

\begin{proof}
Fix $j=1,\ldots,d$ and let $\mu^j$ be as in the statement of the lemma. For (a) we use that there is a natural bijection between $H^1_{0,\{x_j=0\}} (\Omega_j^+) = \{ u \in H^1(\Omega): u|_{\{x_j=0\}} = 0\}$ and
\begin{displaymath}
	H^1_j (\Omega) := \{ u \in H^1(\Omega): u~\text{$j$-antisymmetric} \},
\end{displaymath}
via restriction of functions in $H^1_j(\Omega)$ to $H^1_{0,\{x_j=0\}} (\Omega_j^+)$. Moreover, for any $u \in H^1_j (\Omega)$, its restriction (which we also denote by $u$) has the same Rayleigh quotient,
\begin{displaymath}
	\frac{\int_\Omega |\nabla u|^2}{\int_\Omega |u|^2} = \frac{2\int_{\Omega_j^+} |\nabla u|^2}{2\int_{\Omega_j^+} |u|^2}.
\end{displaymath}
It now follows from the respective variational characterizations of $\mu^j$ and $\mu_{1,j}$ that these two quantities are equal.

For (b), we use that, by Perron--Frobenius theory, as the first eigenvalue of the operator $-\Delta_j^+$, the eigenvalue $\mu_{1,j}$ is necessarily simple, and its eigenfunction $\psi_{1,j}$ does not change sign in $\Omega_j^+$. The identification in (a) clearly extends to equality between all eigenvalues of $-\Delta_j^+$, and the $j$-antisymmetric eigenvalues of $-\Delta_{\rm N}$. In particular, the antisymmetric extension of $\psi_{1,j}$ to $\Omega$ is the only eigenfunction in $H^1_j(\Omega)$ associated with $\mu^j$ (up to scalar multiples).
\end{proof}

We can now complete the proof of Theorem~\ref{thm:hot-spots-II}.

\begin{proof}[Proof of Theorem~\ref{thm:hot-spots-II}]
Fix any $j$, we are interested in the eigenvalue $\mu^j$ and its associated eigenfunction $\psi^j$. Since $\psi^j$ has only one hyperplane of antisymmetry by Lemma~\ref{lem:symm-tech}(b), in particular it is symmetric in every other coordinate hyperplane $\{x_i = 0\}$, $i\neq j$. In particular, via an argument similar to the one used in the proof of that lemma, restricted to the orthant $\Omega \cap \mathcal{O}$ satisfying the assumptions of the theorem, $\mu^j$ is still the first eigenvalue, and $\psi^j$ the (unique) first eigenfunction, of the Laplacian with Dirichlet conditions on the face $\partial (\Omega \cap \mathcal{O}) \cap \{x_j = 0\}$, and Neumann conditions elsewhere. We may now invoke Theorem~\ref{thm:hot-spots-III} directly to obtain that, for each $i$, either $\partial_i \psi^j < 0$ or $\partial_i \psi^j = 0$ or $\partial_i \psi^j > 0$ in $\Omega \cap \mathcal{O}$. By symmetry of $\psi^j$, a corresponding statement must hold in each orthant of $\Omega$, meaning in particular that $\psi^j$ cannot have any interior minima or maxima except possibly a hyperplane of the form $\{x_i = 0\}$ for some $i \neq j$.

Suppose for a contradiction that $\psi^j$ does in fact have an interior extremum, without loss of generality a maximum at some point $p := (0,p_2,\ldots,p_d) \in \{ x_1 = 0 \} \cap \Omega$, without loss of generality in $\Omega_j^+$  (where we keep the notation from Lemma~\ref{lem:antisymm-mixed}). By the maximum principle, necessarily $\psi^j (p) > 0$, so that in particular $\Omega_j^+ = \{ \psi^j > 0 \}$ by our antisymmetry assumptions on $\psi^j$.

Now, since $\psi^j = 0$ on $\{x_j = 0\}$, and by continuity of $\psi^j$, it follows from the trichotomy that $\partial_j\psi^j$ satisfies, in each open orthant, that necessarily $\partial_j \psi^j > 0$ on $\Omega_j^+ \setminus \{ x_1 = 0\}$. Since $\partial_j\psi^j \in C^\infty (\Omega)$, we also have $\partial_j \psi^j \geq 0$ everywhere in $\Omega_j^+$.

The only possibility is that $\psi^j$ is constant on the line segment joining $p$ to the boundary $\partial\Omega$,
\begin{displaymath}
	S_j:= \{ (0,p_2,\ldots,p_{j-1},x_j,p_{j+1},\ldots,p_d) : x_j \in [p_j, x_j^\ast] \},
\end{displaymath}
where $x_j^\ast \in \R$ is such that  $(0,p_2,\ldots,p_{j-1},x_j^\ast,p_{j+1},\ldots,p_d) \in \partial\Omega$. 
We claim this is impossible. Indeed, suppose $\psi^j$ is constant on $S_j$. We may assume without loss of generality that $p_j$ is the minimal value of $x_j$ at which $\partial_j \psi^j$ is zero (noting that $\partial_j \psi^j (0,p_2,\ldots,p_{j-1},x_j,p_{j+1},\ldots,p_d) > 0$ for $x_j$ small enough, since $\psi^j = 0$ on $\{x_j = 0\}$ and $\psi^j(p) > 0$).

Now $\phi := \partial_j \psi^j$ satisfies the eigenvalue equation $-\Delta \phi = \mu^j \phi$ identically in $\Omega$, meaning that $\phi$ is analytic; in particular, its restriction to the line segment $\{x_j \in \R: (0,p_2,\ldots,p_{j-1},x_j,p_{j+1},\ldots,p_d): x_j \in [0,x_j^\ast) \} \simeq [0,x_j^\ast)$ is an analytic one-dimensional function, which is strictly positive on $[0,p_j)$ and identically zero on $[p_j,x_j^\ast)$. This is impossible; hence $\psi^j$ cannot be constant on $S_j$, meaning our original assumption that $\psi^j$ attained a maximum at $p$ cannot hold.
\end{proof}

\section{On the operator $A$ in dimension three}
\label{sec:curl-curl}

We will now revisit the decomposition \eqref{eq:HelmholtzNew} in dimension $d=3$, and also take $\Gamma_{\rm D} = \emptyset$ for simplicity. In this case, it is possible to characterize $A_2$ explicitly in terms of a known operator, a so-called curl curl operator (see, e.g., \cite{HKT12,Z18}).

In fact, on a bounded Lipschitz domain $\Omega \subset \R^3$, consider the self-adjoint operator $C$ in the Hilbert space
\begin{align*}
 \cH_2 := \left\{ u \in L^2 (\Omega)^3 : \diver u = 0, \langle u |_{\partial \Omega}, \nu \rangle = 0 \right\}
\end{align*}
defined via the sesquilinear form
\begin{align*}
 \mathfrak{c} [u, v] & = \int_\Omega \langle \curl u, \curl v \rangle, \quad \dom \mathfrak{c} = \left\{ u \in \cH_2 : \curl u \in L^2 (\Omega)^3 \right\},
\end{align*}
where $\curl u = \nabla \times u = (\frac{\partial u_3}{\partial x_2} - \frac{\partial u_2}{\partial x_3}, \frac{\partial u_1}{\partial x_3} - \frac{\partial u_3}{\partial x_1}, \frac{\partial u_2}{\partial x_1} - \frac{\partial u_1}{\partial x_2})^\top$. A short calculation, which we omit, shows that $C$ is given by
\begin{align*}
 C u & = \curl \curl u, \\ 
 \dom C & = \left\{ u \in \cH_2 : \curl u, \curl \curl u \in L^2 (\Omega)^3, \curl u |_{\partial \Omega} \times \nu = 0 \right\},
\end{align*}
where the boundary condition has to be understood in an appropriate weak sense, see, e.g., \cite[Chapter I, Theorem 2.11]{GR}.
We claim that this operator coincides with $A_2$, the orthogonal projection of $A$ onto $\cH_2$.

\begin{proposition}
\label{prop:form-curl-curl}
Let $\Omega \subset \R^3$ be a bounded Lipschitz domain satisfying Hypothesis~\ref{hyp:dom}, and suppose that $\Gamma_{\rm D} = \emptyset$. Then $A_2 = C$.
\end{proposition}

\begin{proof}
The operator $C$ is self-adjoint in $\cH_2$ with purely discrete spectrum, as follows from the compactness of the embedding of $\dom \mathfrak{c}$ in $\cH_2$, see, e.g., \cite{W74}. As $A_2$ is self-adjoint in $\cH_2$ as well, it suffices to show that each eigenfunction of $C$ is an eigenfunction of $A_2$ corresponding to the same eigenvalue. This follows easily from Lemma \ref{lem:Weingarten}: let $u \in \ker (C - \eta)$ for some $\eta \in \R$. Then 
\begin{align}\label{eq:curlCurlLaplace}
 \eta u = \curl \curl u = - \Delta u
\end{align}
due to $\Delta = \nabla \diver - \curl \curl$ and $\diver u = 0$. It follows from elliptic regularity that $u \in C^\infty (\Omega)$, and regularity up to the boundary, locally near regular boundary points, follows from \cite[Section 4.5]{CDN}. Hence condition (a) of Lemma \ref{lem:Weingarten} is satisfied. For condition (b) it suffices to note that $\{\partial_l u_k - \partial_k u_l\}_{k, l} \nu = \curl u \times \nu$, which vanishes almost everywhere on $\partial \Omega$ for $u \in \dom C$. Condition (c) is void as $\Gamma_{\rm D} = \emptyset$, and condition (d) follows from the eigenvalue equation. Hence it follows from the lemma that $u \in \dom A \cap \cH_2 = \dom A_2$ and $A_2 u = A u = - \Delta u = \eta u$ by \eqref{eq:curlCurlLaplace}. This confirms $A_2 = C$.
\end{proof}

\begin{remark}
In dimension $d = 2$, if $\Omega$ is simply connected then the operator $A_2$ turns out to have a peculiar property: all its eigenfunctions can be obtained as $u = \nabla^\perp \phi := (- \partial_2 \phi, \partial_1 \phi)^\top$, where $\phi$ is any eigenfunction of the Dirichlet Laplacian on $\Omega$. In particular, the spectrum of $A_2$ coincides with that of the Dirichlet Laplacian; cf.\ \cite[Theorem 3.2]{Rprep}. This is no longer true in higher dimensions. In fact, for simply connected $\Omega \subset \R^3$, the spectrum of $A_2 = C$ is strictly positive but not equal to the eigenvalues $\lambda_k$ of the Dirichlet Laplacian; for instance, by \cite[Corollary 3.3]{R24prep}, the first three eigenvalues of $C$ are strictly below $\lambda_1$.
\end{remark}

Inequalities for the eigenvalues of the curl curl operator $C$ have been of interest recently; for instance,  for strictly convex $C^{1,1}$-domains $\Omega \subset \R^3$ it was proved in \cite[Theorem~1.3]{Z18} that $\tau_1 (C) > \mu_2$, i.e., the smallest eigenvalue of $C$ is bounded below by the first positive Neumann Laplacian eigenvalue; this sharpens an earlier result \cite{P15}. We point out that Step 3 of the proof of our Theorem~\ref{thm:hot-spots-III} (in the case $\Gamma_{\rm D} = \emptyset$) in combination with Proposition \ref{prop:form-curl-curl} immediately yields the corresponding statement for (not necessarily convex) lip domains:

\begin{theorem}
Let $\Omega \subset \R^3$ be a bounded lip domain satisfying Hypothesis \ref{hyp:dom}. Then the smallest eigenvalue $\tau_1 (C)$ of the operator $C$ satisfies $\tau_1 (C) > \mu_2$.
\end{theorem}

By combining this result (respectively, \cite[Theorem 1.3]{Z18}) with the results of Section 3, we obtain the following alternative variational principle for the first nonzero eigenvalue of the Neumann Laplacian. Recall that in this case
\begin{align*}
 V = \left\{ u \in H^1 (\Omega)^3 : \langle u |_{\partial \Omega}, \nu \rangle = 0 \right\}.
\end{align*}

\begin{corollary}
Assume that $\Omega \subset \R^3$ is a bounded, strictly convex $C^\infty$-domain or a lip domain satisfying Hypothesis \ref{hyp:dom}. Then 
\begin{align}\label{eq:minMax}
 \mu_2 = \min_{\substack{u \in V \\ u \neq 0}} \frac{\sum_{j = 1}^d \int_\Omega |\nabla u_j|^2 - \int_{\partial \Omega} \langle L u, u\rangle}{\int_\Omega |u|^2},
\end{align}
where $L$ is the shape operator introduced in Section~\ref{sec:operator}. Furthermore, $u \in V$ is a minimizer of \eqref{eq:minMax} if and only if $u = \nabla \psi$ for some $\psi \neq 0$ such that $- \Delta_{\rm N} \psi = \mu_2 \psi$ holds.
\end{corollary}

\end{document}